\documentclass[nointlimits,11pt,oneside]{amsart}
\usepackage{amssymb,cases,enumitem, verbatim}
\usepackage{xcolor}
\usepackage{hyperref}
\usepackage[T1]{fontenc}
\usepackage[utf8]{inputenc}

\hypersetup{
	colorlinks=true,
	linkcolor=blue,
	citecolor=blue,
	filecolor=green,
	urlcolor=cyan,
	bookmarks=true
}

\makeatletter
\renewcommand*{\eqref}[1]{%
	\hyperref[{#1}]{\textup{\tagform@{\ref*{#1}}}}%
}
\makeatother

\setlist[enumerate,1]{label={\textup{(\roman*)}}}
\setlist[enumerate,2]{label={\textup{(\alph*)}}}

\usepackage[%
	a4paper,
	total={16cm,23cm},
	left=3cm, top=3cm,
	marginparsep=2pt]
{geometry}

\theoremstyle{plain}
\newtheorem{theorem}{Theorem}[section]

\newtheorem{corollary}[theorem]{Corollary}
\newtheorem{proposition}[theorem]{Proposition}

\theoremstyle{definition}
\newtheorem{remark}[theorem]{Remark}

\newtheorem{definition}[theorem]{Definition}
\newtheorem{example}[theorem]{Example}

\numberwithin{equation}{section}

\DeclareMathOperator*{\esssup}{ess\,sup}

\DeclareMathOperator*{\loc}{loc}

\DeclareMathOperator*{\supp}{supp}
\DeclareMathOperator*{\dist}{dist}

\newcommand{\norm}[1]{\left\lVert{#1}\right\rVert}
\newcommand{\abs}[1]{\left\lvert {#1}\right\rvert}

\newcommand{\N}{\mathbb{N}}
\newcommand{\R}{\mathcal{R}}
\newcommand{\M}{\mathcal{M}}

\begin{document}

\title{Amalgam approach to compactness in (quasi-)Banach function spaces}
\author{Dalimil Pe\v{s}a}
\address{Dalimil Pe\v{s}a, Department of Automation and Mathematics, Faculty of Electrical Engineering and Informatics, University of Pardubice, Studentská 95, 532 10 Pardubice, Czech Republic}
\email{dalimil.pesa@upce.cz}
\urladdr{0000-0001-6638-0913}

\author{Ivan Kotalík}
\address{Ivan Kotalík, Department of Mathematical Analysis, Faculty of Mathematics and
Physics, Charles University, Sokolovsk\'a~83,
186~75 Praha~8, Czech Republic}
\email{ivan.kotalik690@student.cuni.cz}
\urladdr{0009-0009-0128-2111}


%\thanks{This research was supported by the grant...}

\begin{abstract}
We present amalgam-type characterisations of compactness in quasi-Banach function spaces. We further treat, in the same spirit, several related topics, namely convergence, (uniform) absolute continuity of the norm, and almost-compact embeddings. We also treat the case of rearrangement-invariant spaces and show that most of our conditions can be formulated in terms of the non-increasing rearrangement and the representation space.

\end{abstract}

\date{\today}

\maketitle

%I now know what this does.
\makeatletter
   \providecommand\@dotsep{2}
\makeatother
%\listoftodos\relax

\section{Introduction}

This paper is dedicated to examining compactness in quasi-Banach function spaces and several related topics from an amalgam point of view. The crux of the amalgam approach is to examine the ``local'' behaviour of functions separately from their ``global'' behaviour. Intuitively speaking, ``local'' behaviour is the behaviour on sets that are in some sense small, i.e.~the blow-ups of functions, while ``global'' behaviour is the behaviour ``near infinity'', i.e.~the decay of functions. The usefulness of this way of thinking in various settings has now been evident for a century, starting with the pivotal works of Wiener (\cite{Wiener26}, \cite{Wiener32}, \cite{Wiener33}) which were then followed by many others. For examples outside the setting of (quasi-)Banach function spaces see e.g.~the papers \cite{BusbySmith81}, \cite{Cooper60}, \cite{EdwardsHewitt77}, \cite{FeichtingerGrochenig92}, \cite{JakimovskiRusell80}, \cite{Szeptycki67} and also the surveys \cite{Feichtinger90}, \cite{Feichtinger92}, and \cite{Holland75}; for examples inside the said setting see e.g.~\cite{AlbericoCianchi2018}, \cite{Bathory18}, \cite{BennettRudnick80}, \cite{MusilovaNekvinda25}, \cite{OpicPick99}, and \cite{Vybiral07}. There are also several abstract frameworks available that allow for rigorous application of the amalgam ideas without resorting to ad~hoc constructions. We wish to mention primarily the now-classical Wiener amalgam spaces introduced in their full generality in \cite{Feichtinger83} and the more recent Wiener--Luxemburg amalgam spaces that were designed from the ground up to be compatible with the class of (rearrangement-invariant) quasi-Banach function spaces (i.e.~the setting of this paper) and that were introduced in \cite{Pesa22} (for the rearrangement-invariant case) and \cite{KotalikAmalgams} (for the general case).

As for the concept of compactness in general, we believe its importance and usefulness to be well-known, so let us restrict ourselves to discussing it in the context of quasi-Banach function spaces. In this setting, it can be shown that (roughly speaking) a subset $A$ of the absolutely continuous part of a space $X$ is precompact if and only if the following two conditions are satisfied:
\begin{itemize}[leftmargin=5mm]
    \item $A$ is compact in $\mu_{\loc}$, the topology of convergence in measure on sets of finite measure,
    \item $A$ has uniformly absolutely continuous (quasi)norm.
\end{itemize}
We refer to Section~\ref{SectionQBFS} for the relevant definitions and to Theorem~\ref{Thm:OldCharComp} for the precise statement. This result has its roots in \cite{LuxemburgZaanen62} where it was obtained for Banach function spaces, the full version for quasi-Banach function spaces has been obtained quite recently in \cite[Theorem~3.17]{CaetanoGogatishvili16}. This characterisation has found numerous applications, e.g.~\cite{Pustylnik06}, \cite{KermanPick08}, \cite{RafeiroSamko08}, and \cite{Slavikova15}. The purpose of this paper is to provide an amalgam version of the second condition, i.e.~to replace uniform absolute continuity of the (quasi)norm by separate conditions that provide uniform control of blow-ups and decay, respectively. Let us now present a simplified version of our main characterisation.

\begin{theorem} \label{Thm:CharCompLite}
    Let $(\mathcal{R}, \mu)$ be a $\sigma$-finite measure space, $\norm{\cdot}_X$ be a quasi-Banach function norm over this measure space, and let $X$ be the corresponding quasi-Banach function space. Let $A \subseteq X$ consist of functions having absolutely continuous quasinorm. Then the following statements are equivalent:
    \begin{enumerate}
        \item $A$ is precompact in $(X, \lVert \cdot \rVert_X)$.
        \item  The following three conditions are all satisfied:
        \begin{enumerate}\label{Thm:CharCompLite_ii}
            \item $A$ is precompact with respect to $\mu_{\loc}$, the topology of convergence in measure on sets of finite measure;
            \item \label{Thm:CharCompLite_iib} $$\lim_{n \to \infty} \sup_{\substack{E\subseteq \mathcal{R} \\ \mu(E)\leq n^{-1}}} \sup_{f \in A} \norm{f \chi_{E}}_X = 0;$$
            \item \label{Thm:CharCompLite_iic} $$\lim_{n \to \infty} \inf_{\substack{E\subseteq \mathcal{R}  \\ \mu(E)\leq n}} \sup_{f \in A} \norm{f \chi_{\mathcal{R} \setminus E}}_X = 0.$$
        \end{enumerate}
    \end{enumerate} 
\end{theorem}
The local condition in \ref{Thm:CharCompLite_ii}\ref{Thm:CharCompLite_iib} is as expected, given that the very same condition has been used before to characterise uniform absolute continuity of the quasinorm on spaces of finite measure, e.g.~in  \cite{RafeiroSamko08}. On the other hand, the global condition in \ref{Thm:CharCompLite_ii}\ref{Thm:CharCompLite_iic} is, to the best of our knowledge, new. However, we want to point out that this condition is closely related and directly inspired by the construction used in \cite{MihulaPandy25} to define the relation $\substack{\stackrel{*}{\hookrightarrow} \\ \infty}$ (although there is a fundamental difference between \ref{Thm:CharCompLite_iic} and the construction in said paper).

As it tuns out, Theorem~\ref{Thm:CharCompLite} is just one of a series of related results that appear quite naturally when working with the amalgam-type conditions. Similarly to the classical approach, a characterisation of compactness is closely related to the analogous characterisation of convergence, but thanks to the amalgam approach we can go deeper. Namely, we observe that the global condition \ref{Thm:CharCompLite_iic} may be replaced in the characterisation by the a~priori stronger condition
\begin{enumerate}[label=(c')]
    \item for an arbitrary sequence $\mathcal{R}_n$ of subsets of the underlying measure space $\mathcal{R}$ such that $\mu(\mathcal{R}_n) < \infty$ and $\mathcal{R}_n \nearrow \mathcal{R}$, we have that $$\lim_{n \to \infty} \sup_{f_k \in A} \lVert f_k \chi_{\mathcal{R} \setminus \mathcal{R}_n} \rVert_X = 0.$$
\end{enumerate}
We then consider the three conditions separately and show that each of them individually implies convergence in $(X, \lVert \cdot \rVert_X)$ for sequences that satisfy appropriate a~priori assumptions that are independent of $X$ and are similar in spirit to the convergence in $\mu_{\loc}$. With some caveats, those implications can also be reversed, meaning that the individual conditions are somewhat characterised by this transference of convergence. 

Another venue of inquiry is to consider the condition
\begin{enumerate}[label=(c'')]
    \item \label{c''} $$\lim_{n \to \infty} \sup_{f \in A} \inf_{\substack{E\subseteq \mathcal{R}  \\ \mu(E)\leq n}} \norm{f \chi_{\mathcal{R} \setminus E}}_X = 0.$$
\end{enumerate}
This order of the supremum and infimum is now the same as the one originally used in \cite{MihulaPandy25} and the resulting condition is substantially weaker than \ref{Thm:CharCompLite_iic} as now the excluded set may depend on $f$, not just on $n$. Interestingly enough, this condition can still be used to characterise convergence in $(X, \lVert \cdot \rVert_X)$, provided some reasonable a~priori assumptions on $X$ and the underlying measure space $(\mathcal{R}, \mu)$, but one has to work with the ``full'' convergence in measure, rather than with convergence in $\mu_{\loc}$. What makes this alternative approach especially interesting is the fact that the condition \eqref{c''} is truly amalgam in nature, meaning that it does not arise from an amalgam-type characterisation of a classical condition. This means that this method produces results and insights that would not be easily available without the amalgam approach.

Besides that, we also provide a range of related results that are similar in spirit and work with the above presented amalgam-type conditions or their variations. Namely, we characterise absolute continuity of the quasinorm as well as its uniform version and almost-compact embeddings. Finally, we also translate our work to the setting of rearrangement-invariant spaces, showing that most of our conditions (except, notably, \ref{Thm:CharCompLite_iic} as presented above) can be reformulated in the terms of the non-increasing rearrangement and the representation space. This may be interpreted as a representation-type result for our conditions.

The paper is structured as follows. We first provide the necessary theoretical background in Section~\ref{SecPreliminaries}. We then provide the characterisation of absolute continuity of the quasinorm in Section~\ref{SecACqN}. The core of the paper is Section~\ref{SecAmalgamType} where we state and prove our main results. To be more precise: Section~\ref{SecAmalgamTypeMain} contains the characterisations of convergence and compactness in the spirit of Theorem~\ref{Thm:CharCompLite} and the examination of the individual amalgam-type conditions. Section~\ref{SecAmalgamTypeAlternative} covers the alternative approach to the matter using the condition \ref{c''} and ``full'' convergence in measure. Section~\ref{SecAmalgamTypeUACqN} then presents the characterisations of uniformly absolute continuity of the quasinorm and that of almost-compact embeddings. Finally, Section~\ref{SectionRI} contains our treatment of rearrangement-invariant spaces.

\section{Preliminaries} \label{SecPreliminaries}

We start by introducing some elementary notation. By $(\mathcal{R},\mu)$ or $(\mathcal{S},\nu)$ we denote an arbitrary (totally) $\sigma$-finite measure space. For a $\mu$-measurable set $E\subseteq \mathcal{R}$, we denote its characteristic function as $\chi_E$. By $\M(\mathcal{R},\mu)$ we denote the set of all extended complex-valued $\mu$-measurable functions defined on $\mathcal{R}$. In this set, we identify functions that are equal $\mu$-almost everywhere. Furthermore, we denote the set of all non-negative functions in $\M(\mathcal{R},\mu)$ as $\M_+(\mathcal{R},\mu)$ and the set of functions in $\M(\mathcal{R},\mu)$ that are finite $\mu$-almost everywhere as $\M_0(\mathcal{R},\mu)$. We often shorten $\mu$-almost everywhere to $\mu$-a.e. When the measure space or measure are clear from the context (or irrelevant), we may just write $\mathcal{M}$, $\mathcal{M}_0$, $\mathcal{M}_+$ or a.e. In the particular case of the $n$-dimensional Lebesgue measure, we denote it as $\lambda^n$ or $\lambda$ in the one-dimensional case. We will also use the notation $m$ for the classical counting measure on $\mathbb{N}$; we also always consider $0 \in \mathbb{N}$.

For two topological linear spaces $X,Y \subseteq \mathcal{M}$, we say that $X$ is continuously embedded into $Y$, denoted as $X\hookrightarrow Y$, if we have $X \subseteq Y$ and the identity mapping $I\colon X \to Y$ is continuous.

To be completely precise, we define for $f \in \mathcal{M}$ (i.e.~an equivalence class of functions) the set $\supp f$ as the set
\begin{equation*}
    \supp f := \left \{ x \in \mathcal{R}; \; \left \lvert \overline{f} \right \rvert > 0 \right \},
\end{equation*}
where $\overline{f}$ is an arbitrary representative of $f$. Our arguments will at no point depend in any way on the choice of this representative. Notably, we always have $f = 0$ $\mu$-a.e.~on $\mathcal{R} \setminus \supp f$.

\subsection{Banach function norms and quasinorms}\label{SectionQBFS}

\begin{definition}
    Let $\norm{\cdot}\colon\M(\mathcal{R},\mu)\to [0,\infty]$ be a functional satisfying $\norm{\abs{f}}=\norm{f}$ for all $f\in\M$. We say that $\norm{\cdot}$ is a \emph{Banach function norm} if it satisfies the following for all $a\in \mathbb{C}$, $f,f_n,g\in\M_+$ and $E\subseteq \mathcal{R}$ such that $f_n\nearrow f$ a.e.~and $\mu(E)<\infty$:
    \begin{enumerate}[label=(P\arabic*)]
        \item \label{P1} it is a norm, i.e.~it satisfies
            \begin{enumerate}[label=(\alph*)]
                \item $\norm{af}=\abs{a}\norm{f}$,
                \item $\norm{f}=0$ if and only if $f=0$ a.e.,
                \item $\norm{f+g}\leq \norm{f}+ \norm{g}$,
            \end{enumerate}
        \item \label{P2} it has the lattice property, i.e.~if $f\leq g$ a.e., then $\norm{f}\leq\norm{g}$,
        \item \label{P3} it has the (strong) Fatou property, i.e.~$\norm{f_n}\nearrow\norm{f}$,
        \item \label{P4} it is non-trivial, i.e.~$\norm{\chi_E}<\infty$,
        \item \label{P5} its elements are locally integrable, i.e.~for the set $E$ there exists a constant $C_E>0$ such that $\int_E f \, d\mu \leq C_E \norm{f}$.
    \end{enumerate}
\end{definition}

Naturally, we use Banach function norms to define spaces.

\begin{definition}
    Let $\norm{\cdot}_X$ be a Banach function norm. We define the corresponding \emph{Banach function space} as the collection 
    \begin{equation*}
        X:=\left\{f\in\M(\mathcal{R},\mu); \;  \norm{f}_X < \infty \right\}.
    \end{equation*}
\end{definition}

As the axioms of Banach function norms (and spaces) may prove to be too restrictive, we introduce the much broader class of quasi-Banach function norms (and spaces).

\begin{definition}
    Let $\norm{\cdot}\colon\M(\mathcal{R},\mu)\to [0,\infty]$ be a functional satisfying $\norm{\abs{f}}=\norm{f}$ for all $f\in\M$. We say that $\norm{\cdot}$ is a \emph{quasi-Banach function norm} if it satisfies the axioms \ref{P2}, \ref{P3}, and \ref{P4} of Banach function norms and a weaker version of \ref{P1}, namely we have for all $f, g \in \mathcal{M}_+$ and all $a \in \mathbb{C}$
    \begin{enumerate}[label=(Q\arabic*)]
        \item \label{Q1} $\norm{\cdot}$ is a quasinorm, i.e.~it satisfies
            \begin{enumerate}[label=(\alph*)]
                \item $\norm{af}=\abs{a}\norm{f}$,
                \item $\norm{f}=0$ if and only if $f=0$ a.e.,
                \item \label{Q1c}  there exists a constant $C\geq1$ such that $\norm{f+g}\leq C(\norm{f}+ \norm{g})$.
            \end{enumerate}
    \end{enumerate}
\end{definition}

The smallest possible constant in \ref{Q1}\ref{Q1c} is usually called the modulus of concavity. 

\begin{definition}
    Let $\norm{\cdot}_X$ be a quasi-Banach function norm. We define the corresponding \emph{quasi-Banach function space} as the collection 
    \begin{equation*}
        X:=\left\{f\in\M(\mathcal{R},\mu); \;  \norm{f}_X < \infty \right\}.
    \end{equation*}
\end{definition}

For a detailed treatment of Banach function spaces we refer the reader to \cite[Chapters~1 and 2]{BennettSharpley88}; for an overview of quasi-Banach function spaces we recommend \cite{LoristNieraeth23, MusilovaNekvinda25, NekvindaPesa24}, and the references therein. A somewhat different approaches to the same principal idea are covered in \cite{KreinPetunin82} and \cite{Zaanen67}. Here we focus exclusively on the properties that are directly related to our work.

It will be useful to have a simple notation for a quasi-Banach function space equipped with the corresponding quasinormed topology.

\begin{definition}
    Let $\norm{\cdot}_X$ be a quasi-Banach function norm and let $X$ be the corresponding quasi-Banach function space. By $(X, \lVert \cdot \rVert_X)$ we denote the space $X$ equipped with the quasinormed topology induced by $\lVert \cdot \rVert_X$.
\end{definition}

Our characterisations of convergence and compactness will of course often work with  the space $\mathcal{M}_0$ equipped with the topology of convergence in measure on sets of finite measure.

\begin{definition}
    We shall denote by $(\mathcal{M}_0, \mu_{\loc})$ the space $\mathcal{M}_0(\mathcal{R}, \mu)$ of all (equivalence classes of) complex valued $\mu$-a.e.~finite functions equipped with the topology of convergence in measure on sets of finite measure.
\end{definition}

It is worth noting that both $(X, \lVert \cdot \rVert_X)$ and $(\mathcal{M}_0, \mu_{\loc})$ are completely metrisable topological vector spaces that in general fail to be locally-convex (see the books and papers listed above and the references therein).

The usefulness of $(\mathcal{M}_0, \mu_{\loc})$ originates in the following embedding. This result is classical for the normed case, see e.g.~\cite[Chapter~1, Theorem~1.4]{BennettSharpley88}; its extension to the case of quasi-Banach function spaces was provided in~\cite{NekvindaPesa24}.

\begin{theorem} \label{Thm:EmbToMeasure}
    Let $\norm{\cdot}_X$ be a quasi-Banach function norm and let $X$ be the corresponding quasi-Banach function space. Then 
    \begin{equation*}
        (X, \lVert \cdot \rVert_X) \hookrightarrow (\mathcal{M}_0, \mu_{\loc}).
    \end{equation*}
\end{theorem}

As we know, the quasinorms of characteristic functions of sets having the same measure may vary when working with general quasi-Banach function spaces. Therefore, we cannot simply define a fundamental function as in rearrangement-invariant spaces (see e.g.~\cite[Chapter~2]{BennettSharpley88} for the classical case and all the relevant definitions). However, the concept of extremal fundamental functions that has been introduced recently in \cite{MihulaPandy25} overcomes this limitation.

\begin{definition} \label{Def:FundFunc}
    Let $\norm{\cdot}_X$ be a quasi-Banach function norm and let $X$ be the corresponding quasi-Banach function space. We define the \emph{minimal fundamental function} of $\norm{\cdot}_X$ as 
     \begin{align*}
        \varphi_X^{\min}(t) &:= \inf_{\substack{E \subseteq \mathcal{R} \\ \mu(E)=t}} \norm{\chi_E}_X,  &\text{for } t \text{ in the range of } \mu,
     \end{align*}
     and the \emph{maximal fundamental function} of $\norm{\cdot}_X$ as 
     \begin{align*}
        \varphi_X^{\max}(t) &:= \sup_{\substack{E \subseteq \mathcal{R} \\\mu(E)=t}} \norm{\chi_E}_X, &\text{for } t \text{ in the range of } \mu.
     \end{align*}

     If it holds for some $t \in [0, \infty)$ that $\varphi_X^{\max}(t) = \varphi_X^{\min}(t)$ then we denote this common value as $\varphi_X(t)$ and call it the \emph{fundamental function} of $\lVert \cdot \rVert_X$.
\end{definition}

\begin{remark} \label{Rem:FundMono}
    If the measure space $(\mathcal{R}, \mu)$ is either non-atomic or completely atomic with all atoms having the same measure, then it is rather easy to see that the minimal and maximal fundamental functions are non-decreasing.
\end{remark}

An interesting result is that on non-atomic measure spaces there are uniform bounds on $\varphi_X^{\min}(t)$ and $\varphi_X^{\max}(t)$. This was first shown in~\cite[Lemma 4.4]{Slavikova12} for Banach function spaces, but the original proof, mutatis mutandis, still works. The generalisation to the currently presented form comes from \cite[Proposition 3.2]{MihulaPandy25}; we note that the result is trivial if $\mu(\mathcal{R}) < \infty$ as then $\varphi_X^{\max} \leq \lVert \chi_{\mathcal{R}} \rVert_X < \infty$.

\begin{theorem}\label{lenka4.4}
    Let $(\mathcal{R}, \mu)$ be non-atomic, let $\norm{\cdot}_X$ be a quasi-Banach function norm, and let $X$ be the corresponding quasi-Banach function space. Let $t\in(0,\mu(\mathcal{R}))$. Then $\varphi_X^{\max}(t) < \infty$. Furthermore, if $\norm{\cdot}_X$ satisfies \ref{P5}, then $\varphi_X^{\min}(t) > 0$.
\end{theorem}

Of course, the positivity of $\varphi_X^{\max}(t)$ for $t > 0$ is immediate from the axioms (namely \ref{Q1}). As it turns out, it will be sometimes necessary to restrict ourselves to spaces where $\varphi_X^{\min}(t)$ is also positive for $t>0$.

\begin{definition}
    Let $\norm{\cdot}_X$ be a quasi-Banach function norm and let $X$ be the corresponding quasi-Banach function space. 
    We say that $\norm{\cdot}_X$ is \emph{admissible} if $\varphi_X^{\min}(t)>0$ for all $t>0$ in the range of $\mu$.
\end{definition}

Another property of quasi-Banach function spaces that is deeply related to the topic at hand is the absolute continuity of the quasinorm, which is studied in~\cite[Chapter 1.3]{BennettSharpley88}. We will also need its uniform version. 

\begin{definition}
     Let $\norm{\cdot}_X$ be a quasi-Banach function norm and let $X$ be the corresponding quasi-Banach function space. A function $f\in X$ is said to have \emph{absolutely continuous quasinorm} in $X$ if for every sequence $E_n \subseteq \mathcal{R}$ satisfying $\chi_{E_n}\to 0$ a.e.~it holds that
     \begin{equation*}
        \lim_{n \to \infty} \norm{f\chi_{E_n}}_X = 0.
     \end{equation*}
     We denote the subspace of all functions with absolutely continuous quasinorm in $X$ as $X_a$. If it is that $X=X_a$, we say that the space $X$ has \emph{absolutely continuous quasinorm}. 

     We further say that a set $A \subseteq X$ has \emph{uniformly absolutely continuous quasinorm} in $X$ if for every sequence $E_n \subseteq \mathcal{R}$ satisfying $\chi_{E_n}\to 0$ a.e.~it holds that
     \begin{equation*}
        \lim_{n \to \infty} \sup_{f \in A} \norm{f\chi_{E_n}}_X = 0.
     \end{equation*}
\end{definition}

The following proposition is easily verified, in a manner virtually identical to that of \cite[Chapter~1, Proof of Theorem~3.8]{BennettSharpley88}.

\begin{proposition} \label{PropXaOrdId}
	The set $X_a$ is an order ideal in $X$, i.e.~it is a closed (with respect to the quasinorm $\lVert \cdot \rVert_X$) linear subspace satisfying that whenever $f \in X_a$ and $\lvert g \rvert \leq \lvert f \rvert$ $\mu$-a.e.~then also $g \in X_a$.
\end{proposition}

The relationship to compactness is captured by the result obtained in \cite[Theorem~3.17]{CaetanoGogatishvili16} (with the case of Banach function spaces being known as early as \cite{LuxemburgZaanen62}).

\begin{theorem} \label{Thm:OldCharComp}
    Let $\norm{\cdot}_X$ be a quasi-Banach function norm and let $X$ be the corresponding quasi-Banach function space. Let $A \subseteq X_a$. Then $A$ is precompact in $(X, \lVert \cdot \rVert_X)$ if and only if it is precompact in $(\mathcal{M}_0, \mu_{\loc} )$ and has uniformly absolutely continuous quasinorm.
\end{theorem}

\subsection{Non-increasing rearrangement}

In Section~\ref{SectionRI}, we will treat the case of rearrangement-invariant spaces, whence we now provide the necessary background. As our focus is the general case, we will keep the exposition to a minimum. To the reader interested in a more in-depth exposition, we recommend the books~\cite{BennettSharpley88} or \cite{KreinPetunin82} and, for the case of quasinorms, the papers~\cite{MusilovaNekvinda25} and~\cite{NekvindaPesa24} and the references therein.

\begin{definition}
    Let $f\in \M(\mathcal{R}, \mu)$. We define the \emph{distribution function} of $f$ as the function
    \begin{equation*}
        f_*(s):= \mu(\{x\in \mathcal{R}; \; \abs{f(x)}>s\}), \ s\in[0,\infty).
    \end{equation*}
    We further define the \emph{non-increasing rearrangement} of $f$ as the function
    \begin{equation*}
        f^*(t):= \inf\{s\in [0,\infty); \; f_*(s)\leq t\}, \ t\in[0,\infty).
    \end{equation*}
\end{definition}

The basic properties of these objects can be found in~\cite[Chapter 2, Proposition 1.3]{BennettSharpley88} and~\cite[Chapter 2, Proposition 1.7]{BennettSharpley88}. We consider these properties well-known enough to omit them and use them without explicit reference.
\begin{definition}
    We say that two functions $f\in\M(\mathcal{R}, \mu)$ and $g\in \M(S,\nu)$ are \emph{equimeasurable} if $f_*=g_*$.
\end{definition}

\begin{definition}
    Let $\norm{\cdot}_X$ be a quasi-Banach function norm and $X$ its corresponding quasi-Banach function space. If for every $f,g\in\M(\mathcal{R}, \mu)$ we have
    \begin{enumerate}[label=(P\arabic*)] \setcounter{enumi}{5}
        \item \label{P6} $\norm{f}_X=\norm{g}_X$ whenever $f_*=g_*$,
    \end{enumerate}
    we say that $\norm{\cdot}_X$ and $X$ are \emph{rearrangement-invariant}, abbreviated to r.i.
\end{definition}

As it turns out, r.i.~spaces work well only if we assume that the underlying measure space has some nice properties. For our purposes, the following definition will be sufficient:

\begin{definition}
    We say that a $\sigma$-finite measure space $(\mathcal{R}, \mu)$ is resonant if it is either non-atomic or completely atomic with atoms having equal measure.
\end{definition}

We note that usually a different definition is provided and the formulation we used is then proved to be a characterisation, see e.g.~\cite[Chapter~2, Section~2]{BennettSharpley88}.

The property \ref{P6} has deep consequences, one of which is the following representation theorem. Essentially, we can replace an r.i.~quasinorm over any resonant measure space with a quasinorm over $([0, \infty), \lambda)$. The original result for Banach function norms is the classical Luxemburg representation theorem; see \cite{Luxemburg67} for the original result and \cite[Chapter 2, Theorem 4.10]{BennettSharpley88} for a modern presentation. The version for quasi-Banach function norms originated in~\cite[Theorem 3.1]{MusilovaNekvinda25} with a slight modification to the below-presented form provided in \cite[Section~3]{Pesa25}. 

\begin{theorem}
    Let $(\mathcal{R}, \mu)$ be a resonant measure space and let $\norm{\cdot}_X$ be an r.i.~quasi-Banach function norm on $\M(\mathcal{R}, \mu)$. Then there exists an r.i.~quasi-Banach function norm $\norm{\cdot}_{\overline{X}}$ on $\M([0,\infty), \lambda)$ such that for all $f\in\M(\mathcal{R}, \mu)$ it holds $\norm{f}_X=\norm{f^*}_{\overline{X}}$.
    
    Furthermore, $\norm{\cdot}_{\overline{X}}$ has the property (P5) whenever $\norm{\cdot}_X$ does. Finally, if $(\mathcal{R}, \mu)$ is non-atomic and of infinite measure, then $\norm{\cdot}_{\overline{X}}$ is uniquely determined.
\end{theorem}

Note that in the cases when the representation quasinorm is not uniquely determined, some results may depend on the precise choice. For that reason, we will always use the symbol $\norm{\cdot}_{\overline{X}}$ to denote the functional constructed in \cite[Definition~3.1]{Pesa25}.

In rearrangement-invariant spaces, the characteristic functions of sets of given fixed finite measure all have the same norm. Hence, the functions from Definition~\ref{Def:FundFunc} satisfy $\varphi_X^{\max}(t) = \varphi_X^{\min}(t) = \varphi_X(t)$ for every $t$ in the range of $\mu$ and we have $\varphi_X(t)=\norm{\chi_{[0,t]}}_{\overline{X}}$. The right-hand side of this formula can then be used to extend $\varphi_X$ onto the entire set $[0, \infty)$.

\section{Absolute continuity of the quasinorm} \label{SecACqN}

As compactness is closely related to the absolute continuity of the quasinorm (see Theorem~\ref{Thm:OldCharComp}), it seems that providing an amalgam-type characterisation of this property would be a good place to start. We first have to introduce two functionals that we will employ throughout the paper. We would like to point out again that their construction is inspired by the work conducted in \cite{MihulaPandy25}. Furthermore, the utility of these functionals seems to exceed the scope of this paper, as illustrated e.g.~by \cite{KotalikAmalgams}.

\begin{definition}\label{Def:maximal_minimal_norm}
    Let $\norm{\cdot}_X$ be a quasi-Banach function norm and let $X$ be the corresponding quasi-Banach function space. We define the \emph{maximal local norm} of $f\in X$
    \begin{align*}
        \sigma_Xf(t) &:=\sup_{\substack{E\subseteq \mathcal{R}  \\ \mu(E)\leq t}} \norm{f \chi_{E}}_X, &\text{for } t \in (0, \infty),
    \end{align*}
    and the \emph{minimal global norm} of $f\in X$
    \begin{align*}
        \delta_Xf(t) &:=\inf_{\substack{E\subseteq \mathcal{R}  \\ \mu(E)\leq t}} \norm{f \chi_{\mathcal{R} \setminus E}}_X, &\text{for } t \in (0, \infty).
    \end{align*}
\end{definition}

\begin{remark} \label{Rem:SpecialMu}
    Clearly, if $(\mathcal{R}, \mu)$ is such that
    \begin{equation*}
        \alpha := \inf_{\substack{E\subseteq \mathcal{R}  \\ E \neq \emptyset}} \mu (E) > 0,
    \end{equation*}
    (such as e.g.~in $(\mathbb{N}, m)$) then $\sigma_Xf(t) = 0$ for $t < \alpha$, regardless of $f$. Similarly, if $\mu(\mathcal{R}) < \infty$, then $\delta_Xf(t) = 0$ for $t > \mu(\mathcal{R})$, regardless of $f$.
\end{remark} 

\begin{theorem}\label{Thm:ac_char}
    Let $\norm{\cdot}_X$ be a quasi-Banach function norm and let $X$ be the corresponding quasi-Banach function space. Let $f\in X$. Then $f\in X_a$ if and only if the following two conditions hold:
    \begin{equation}\label{local}
        \lim_{n\rightarrow\infty} \sigma_Xf\left(n^{-1}\right) = 0,
    \end{equation}
    and
    \begin{equation}\label{global}
        \lim_{n\rightarrow\infty} \delta_Xf(n) = 0.
    \end{equation}
\end{theorem}

It is worth pointing out that we make no assumptions on the underlying measure space (except $\sigma$-finiteness). We also note that a similar result for r.i.~quasi-Banach function spaces over a resonant measure space has previously been obtained in \cite[Proposition~4.5]{Pesa25}; as we will see in Section~\ref{SectionRI}, this older result is a special case of the Theorem at hand. 
Finally, similarly to what we will show in Theorem~\ref{Thm:CharConv} \ref{Thm:CharConv_iii}, the characterisation still holds if \eqref{global} is replaced by a certain stronger condition. However, we do not see an immediate application for this formulation, and so we decided to not include it for the sake of simplicity of presentation.

\begin{proof}
    We start with the necessity for~\eqref{local}. Let $\varepsilon>0$. 
    For every $n\in\N$, we find a set $E_n\subseteq \mathcal{R}$ such that $\mu(E_n)\leq n^{-1}$ and 
    \begin{equation*}
        \sigma_Xf\left(n^{-1}\right)  -\varepsilon < \norm{f \chi_{E_n}}_X.
    \end{equation*}
    We have $\mu(E_n)\rightarrow 0$. Hence, there exists a subsequence $E_{n_{k}}$ such that $\chi_{E_{n_{k}}}\rightarrow 0$ a.e. Because $f\in X_a$, we get that there exists $k_0\in\N$ such that for every $k\geq k_0$,
    \begin{equation*}
        \norm{f \chi_{E_{n_{k}}}}_X < \varepsilon.
    \end{equation*}
    Thus, for every $k\geq k_0$, we get
    \begin{equation*}
        \sigma_Xf\left(n_k^{-1}\right)< \norm{f \chi_{E_{n_k}}}_X + \varepsilon < 2 \varepsilon,
    \end{equation*}
    which is enough due to the monotonicity of our sequence. 
    
    For the necessity of~\eqref{global}, we use the $\sigma$-finiteness of $(\mathcal{R},\mu)$ to find a sequence of sets $E_n\subseteq \mathcal{R}$ such that $E_n \nearrow \mathcal{R}$ and $\mu(E_n)<n$ for every $n\in\N$. Then $\chi_{\mathcal{R}\setminus E_n}\rightarrow 0$ a.e. Let $\varepsilon >0$. Because $f\in X_a$, we find $n_0\in\N$ such that for every $n\geq n_0$ we have
    \begin{equation*}
        \norm{f \chi_{\mathcal{R}\setminus E_n}}_X<\varepsilon.
    \end{equation*}
    We obtain that for every $n\in\N, n\geq n_0$,
    \begin{equation*}
        \delta_Xf(n) \leq \norm{f \chi_{\mathcal{R}\setminus E_n}}_X < \varepsilon.
    \end{equation*} 
    
    For sufficiency, let $\varepsilon>0$ and let $\chi_{E_n}\rightarrow 0$ a.e.~for some sequence $E_n\subseteq \mathcal{R}$. By~\eqref{global}, we find $n_0\in\N$ such that 
    \begin{equation*}
        \delta_Xf(n_0) < \varepsilon.
    \end{equation*}
    Hence, there exists a set $E_0$ such that $\mu(E_0) \leq n_0$ and 
    \begin{equation}\label{E0}
         \norm{f \chi_{\mathcal{R}\setminus E_0}}_X<\varepsilon.
    \end{equation}
    We denote $\widetilde{E}_n = E_n \cap E_0$ and note that $\mu(\widetilde{E}_n) \rightarrow 0$ by the dominated convergence theorem for the Lebesgue integral. Next, by~\eqref{local}, we find $n_1\in\N$ such that for every set $E\subseteq \mathcal{R}$, $\mu(E)< n_1^{-1}$, we have 
    \begin{equation*}
         \norm{f \chi_{E}}_X<\varepsilon.
    \end{equation*}
    Because $\mu(\widetilde{E}_n) \rightarrow 0$, we find $n_2\in\N$, $n_2 \geq n_0$, such that for every $n\geq n_2$ we have $\mu(\widetilde{E}_n) < n_1^{-1}$ and by the previous equation
    \begin{equation}\label{EK}
        \norm{f \chi_{\widetilde{E}_n}}_X<\varepsilon,
    \end{equation}
    for every $n\geq n_2$. Using~\eqref{E0} and~\eqref{EK}, we obtain for every $n\geq n_2$, that
    \begin{equation*}
        \norm{f \chi_{E_n}}_X \lesssim \norm{f\chi_{\widetilde{E}_k}}_X+\norm{f\chi_{\mathcal{R}\setminus E_0}}_X < 2 \varepsilon.
    \end{equation*}
\end{proof}

This theorem reveals a simple characterisation for when the simple functions have absolutely continuous quasinorm.

\begin{corollary}
    Let $\norm{\cdot}_X$ be a quasi-Banach function norm and let $X$ be the corresponding quasi-Banach function space. Then $X_a$ contains the simple functions if and only if $\lim_{t\to 0+} \varphi_X^{\max}(t)=0$.
\end{corollary}

\begin{proof}
    As $X_a$ is an order ideal, it suffices to show the characterisation for the characteristic functions of sets of finite measure. 
    
    Let $G\subseteq \R$, $0<\mu(G)<\infty$ be arbitrary. It is easy to see that $\delta_X \chi_G (n)=0$ for $n>\mu(G)$ and for any $n\in \N$, it holds
    \begin{equation*}
        \sigma_X \chi_G \left (n^{-1} \right) = \sup_{\substack{E\subseteq \mathcal{R}  \\ \mu(E)\leq n^{-1}}} \norm{\chi_{E\cap G}}_X \leq \sup_{\substack{E\subseteq \mathcal{R}  \\ \mu(E)\leq n^{-1}}} \norm{\chi_{E}}_X \leq \sup_{t\in(0,n^{-1})} \varphi_X^{\max}\left( t\right) \to 0.
    \end{equation*}
    Theorem~\ref{Thm:ac_char} now shows that $\chi_G\in X_a$.

    On the other hand, if there is some $G \subseteq \mathcal{R}$ such that $0<\mu(G)<\infty$ while $\chi_G \notin X_a$, then there is a sequence $E_n$ of subsets of $\mathcal{R}$ such that $\chi_{E_n} \to 0$ $\mu$-a.e.~while
    \begin{equation*}
        \limsup_{n \to \infty} \; \lVert \chi_{G \cap E_n} \rVert_X > 0.
    \end{equation*}
    As clearly $\mu(G\cap E_n) \to 0$, we get
    \begin{equation*}
        \limsup_{t \to 0_+} \; \varphi_X^{\max}(t) \geq \limsup_{n \to \infty} \; \lVert \chi_{G \cap E_n} \rVert_X > 0.
    \end{equation*}
    
\end{proof}

\section{Amalgam-type characterisations of convergence and compactness} \label{SecAmalgamType}

\subsection{The main amalgam-type characterisations} \label{SecAmalgamTypeMain}

To prove our main amalgam-type characterisations of convergence and compactness, we will require uniform versions of the functionals introduced in Definition~\ref{Def:maximal_minimal_norm}.

\begin{definition}\label{Def:uniform_maximal_minimal_norm}
    Let $\norm{\cdot}_X$ be a quasi-Banach function norm and let $X$ be the corresponding quasi-Banach function space. We define the \emph{uniform maximal local norm} of $A \subseteq X$
    \begin{align*}
        \sigma_X A(t)&:=\sup_{\substack{E\subseteq \mathcal{R} \\ \mu(E)\leq t}} \sup_{f \in A} \norm{f \chi_{E}}_X, &\text{for } t \in (0, \infty).
    \end{align*}
    and the \emph{uniform minimal global norm} of $A \subseteq X$
    \begin{align*}
        \delta_X A(t)&:=\inf_{\substack{E\subseteq \mathcal{R}  \\ \mu(E)\leq t}} \sup_{f \in A} \norm{f \chi_{\mathcal{R} \setminus E}}_X, &\text{for } t \in (0, \infty).
    \end{align*}
\end{definition}

\begin{remark}
    As suprema can be interchanged, we get
    \begin{align*}
        \sigma_X A(t)&=\sup_{f \in A} \sigma_X f(t), &\text{for } t \in (0, \infty).
    \end{align*}
    On the contrary, we cannot interchange a supremum with an infimum, so an analogous formula fails to hold for $\delta_X$. This is rather important, as it implies that
    \begin{align*}
        \lim_{n \to \infty} \delta_X A(n) &= 0, \\
        \lim_{n \to \infty} \sup_{g \in A} \delta_X g(n) &= 0,
    \end{align*}
    are two distinct conditions, with the former being stronger than the latter.
\end{remark}

In the light of Theorem~\ref{Thm:OldCharComp}, it would be natural to characterise compactness through providing an amalgam characterisation of the uniform absolute continuity of the quasinorm. Such an approach is indeed possible (and we present such a characterisation in Theorem~\ref{Thm:uac_char} below), but we decided to provide a direct proof of the characterisation, i.e.~one not requiring Theorem~\ref{Thm:OldCharComp}, because it naturally leads to another form of the characterising conditions that we find to be quite interesting.

The first step is a characterisation of convergence.

\begin{theorem} \label{Thm:CharConv}
    Let $\norm{\cdot}_X$ be a quasi-Banach function norm and let $X$ be the corresponding quasi-Banach function space. Let $f_k$ be a sequence of functions in $X_a$ and denote $A = \{ f_k; \; k \in \mathbb{N} \}$. Then the following statements are equivalent:
    \begin{enumerate}
        \item \label{Thm:CharConv_i} There is some $f \in X$ such that $f_k \to f$ in $(X, \lVert \cdot \rVert_X)$.
        \item \label{Thm:CharConv_ii} The following three conditions are all satisfied:
        \begin{enumerate}
            \item \label{Thm:CharConv_iia} There is some $f \in \mathcal{M}_0$ such that $f_k \to f$ in $(\mathcal{M}_0, \mu_{\loc} )$,
            \item \label{Thm:CharConv_iib} $\lim_{n \to \infty} \sigma_X A(n^{-1}) = 0$,
            \item \label{Thm:CharConv_iic} $\lim_{n \to \infty} \delta_X A(n) = 0$.
        \end{enumerate}
        \item \label{Thm:CharConv_iii} The following three conditions are all satisfied:
        \begin{enumerate}
            \item \label{Thm:CharConv_iiia} There is some $f \in \mathcal{M}_0$ such that $f_k \to f$ in $(\mathcal{M}_0, \mu_{\loc} )$,
            \item \label{Thm:CharConv_iiib} $\lim_{n \to \infty} \sigma_X A(n^{-1}) = 0$,
            \item \label{Thm:CharConv_iiic} for an arbitrary sequence $\mathcal{R}_n$ of subsets of $\mathcal{R}$ such that $\mu(\mathcal{R}_n) < \infty$ and $\mathcal{R}_n \nearrow \mathcal{R}$, we have that $\lim_{n \to \infty} \sup_{f_k \in A} \lVert f_k \chi_{\mathcal{R} \setminus \mathcal{R}_n} \rVert_X = 0$.
        \end{enumerate}
    \end{enumerate} 
\end{theorem}

It is immediate from Theorem~\ref{Thm:EmbToMeasure} that if either of the three conditions holds, then the limits in $(X, \lVert \cdot \rVert_X)$ and in $(\mathcal{M}_0, \mu_{\loc} )$ coincide. This is also probably a good place to note that the implication from \ref{Thm:CharConv_ii} to \ref{Thm:CharConv_i} does not require that the sequence is a~priori known to be in $X_a$, but we need this a~priori assumption for the converse implication, as e.g.~constant sequences of functions in $X \setminus X_a$ also converge but clearly fail our conditions (as per Theorem~\ref{Thm:ac_char}). Hence, this assumption will be present in characterisations of convergence or compactness, but not in results containing only sufficient conditions.

\begin{proof}
    It is clear that \ref{Thm:CharConv_iii} implies \ref{Thm:CharConv_ii}. Let us assume that \ref{Thm:CharConv_ii} holds and fix $\varepsilon > 0$. Using \ref{Thm:CharConv_ii}\ref{Thm:CharConv_iic}, we find some $E_0 \subseteq \mathcal{R}$ with $\mu(E_0) < \infty$ such that
    \begin{align*}
        \sup_{f_k \in A} \norm{f_k \chi_{\mathcal{R}\setminus E_0}}_X &< \varepsilon. \\
    \end{align*}
    As there is a subsequence of $f_k$ converging to $f$ $\mu$-a.e., we may apply the Fatou's lemma (\cite[Theorem~3.4]{NekvindaPesa24}) to get that also
    \begin{equation*}
        \norm{f \chi_{\mathcal{R}\setminus E_0}}_X < \varepsilon.
    \end{equation*}    
    
    We also find via \ref{Thm:CharConv_ii}\ref{Thm:CharConv_iib} some $n \in \mathbb{N}$ such that $\sigma_X A(n^{-1}) < \varepsilon$. Then we apply the convergence $f_k \to f$ in $(\mathcal{M}_0, \mu_{\loc} )$ onto the set $E_0$ to find some $k_0 \in \mathbb{N}$ such that the sets
    \begin{align*}
        E_k &:=\left \{ x \in E_0 ; \; \lvert f(x) - f_k(x) \rvert > \varepsilon \lVert \chi_{E_0} \rVert_X^{-1} \right \}
    \end{align*}
    satisfy $\mu \left ( E_k \right ) < n^{-1}$ for every $k \geq k_0$. Next, if we fix any $\widetilde{k} \geq k_0$ and the corresponding set $E_{\widetilde{k}}$, then the a.e.~convergence of an appropriate subsequence of $f_k$ and the Fatou's lemma (\cite[Theorem~3.4]{NekvindaPesa24}) show that also
    \begin{equation*}
        \norm{f \chi_{E_{\widetilde{k}}}}_X < \varepsilon.
    \end{equation*}
    
    Finally, we get for every $k \geq k_0$
    \begin{equation*}
    \begin{split}
        \lVert f - f_k \rVert_X &\lesssim \lVert f \chi_{E_k} \rVert_X + \lVert f_k \chi_{E_k} \rVert_X + \lVert (f - f_k) \chi_{E_0 \setminus E_k} \rVert_X +  \lVert f\chi_{\mathcal{R} \setminus E_0} \rVert_X + \lVert f_k \chi_{\mathcal{R} \setminus E_0} \rVert_X \\
        &< 5 \varepsilon.
    \end{split}
    \end{equation*}
    This also shows that $f \in X$.

    It remains to show that \ref{Thm:CharConv_i} implies \ref{Thm:CharConv_iii}. Fix again $\varepsilon > 0$. Theorem~\ref{Thm:EmbToMeasure} shows that $f_k \to f$ in $(\mathcal{M}_0, \mu_{\loc} )$. As $X_a$ is closed in $(X, \lVert \cdot \rVert_X)$ (Proposition~\ref{PropXaOrdId}), we get $f \in X_a$. Let $k_0 \in \mathbb{N}$ be such that we have $\lVert f - f_k \rVert_X < \varepsilon$ for all $k \geq k_0$. Applying Theorem~\ref{Thm:ac_char} on $f$ and $f_k$, $k < k_0$, we may find $n_0 \in \mathbb{N}$ such that
    \begin{align*}
        \sigma_X f(n^{-1}) &< \varepsilon &\text{for all } n \geq n_0, \\
        \sigma_X f_k(n^{-1}) &< \varepsilon &\text{for all } k < k_0 \text{ and all } n \geq n_0.
    \end{align*}
    As for $k \geq k_0$, we compute for any $E \subseteq \mathcal{R}$ with $\mu(E) < n^{-1}$
    \begin{equation*}
        \lVert f_k \chi_E \rVert_X \lesssim \lVert (f - f_k) \chi_E \rVert_X + \lVert f \chi_E \rVert_X \lesssim \varepsilon,
    \end{equation*}
    whence $\sigma_X A(n^{-1}) \lesssim \varepsilon$ for all $n \geq n_0$. 
    
    Similarly, by applying Theorem~\ref{Thm:ac_char} on $f$ and $f_k$, $k < k_0$, and then taking the union of the obtained sets, we get some $E_0 \subseteq \mathcal{R}$ such that $\mu(E_0) < \infty$ and
    \begin{align*}
        \lVert f \chi_{\mathcal{R} \setminus E_0} \rVert_X
        &< \varepsilon, \\
        \lVert f_k \chi_{\mathcal{R} \setminus E_0} \rVert_X &< \varepsilon &\text{for all } k < k_0.
    \end{align*}
    We now compute for $k \geq k_0$
    \begin{equation*}
        \lVert f_k \chi_{\mathcal{R} \setminus E_0} \rVert_X \lesssim \lVert (f - f_k) \chi_{\mathcal{R} \setminus E_0} \rVert_X + \lVert f \chi_{\mathcal{R} \setminus E_0} \rVert_X \lesssim \varepsilon.
    \end{equation*}
    This would suffice for \ref{Thm:CharConv_ii}\ref{Thm:CharConv_iic}, but we must show \ref{Thm:CharConv_iii}\ref{Thm:CharConv_iiic}. To this end, fix arbitrary $\mathcal{R}_n$ as in the statement and observe that there is some $n_1 \in \mathbb{N}$ such that $\mu(E_0 \setminus \mathcal{R}_{n_1}) < n_0^{-1}$, where $n_0$ is the number obtained in the step concerning $\sigma_X A$. Then we have for $n \geq n_1$ and every $k \in \mathbb{N}$
    \begin{equation*}
    \begin{split}
        \lVert f_k \chi_{\mathcal{R} \setminus \mathcal{R}_n} \rVert_X &\leq \lVert f_k \chi_{\mathcal{R} \setminus \mathcal{R}_{n_1}} \rVert_X \\
        &\lesssim \lVert f_k \chi_{\mathcal{R} \setminus E_0} \rVert_X + \lVert f_k \chi_{E_0 \setminus \mathcal{R}_{n_1}} \rVert_X \\
        &\lesssim \varepsilon + \sigma_X A(n^{-1}) \\
        &\lesssim \varepsilon.
    \end{split}
    \end{equation*}    
\end{proof} 

Note that the condition \ref{Thm:CharConv_iii}\ref{Thm:CharConv_iiic} in the previous Theorem is as of itself strictly stronger than \ref{Thm:CharConv_ii}\ref{Thm:CharConv_iic} (even for an individual function, consider e.g.~$f = \min \{1, \frac{1}{x} \}$ in $L^{\infty}((0, \infty))$; it is the presence of the common part \ref{Thm:CharConv_iib} that makes \ref{Thm:CharConv_ii} and \ref{Thm:CharConv_iii} equivalent. Before discussing the purpose of introducing this stronger condition we have to present the promised characterisation of compactness.

\begin{theorem} \label{Thm:CharComp}
    Let $\norm{\cdot}_X$ be a quasi-Banach function norm and let $X$ be the corresponding quasi-Banach function space. Let $A \subseteq X_a$. Then the following statements are equivalent:
    \begin{enumerate}
        \item \label{Thm:CharComp_i} $A$ is precompact in $(X, \lVert \cdot \rVert_X)$.
        \item \label{Thm:CharComp_ii} The following three conditions are all satisfied:
        \begin{enumerate}
            \item \label{Thm:CharComp_iia} $A$ is precompact in $(\mathcal{M}_0, \mu_{\loc} )$,
            \item \label{Thm:CharComp_iib} $\lim_{n \to \infty} \sigma_X A(n^{-1}) = 0$,
            \item \label{Thm:CharComp_iic} $\lim_{n \to \infty} \delta_X A(n) = 0$.
        \end{enumerate}
        \item \label{Thm:CharComp_iii} The following three conditions are all satisfied:
        \begin{enumerate}
            \item \label{Thm:CharComp_iiia} $A$ is precompact in $(\mathcal{M}_0, \mu_{\loc} )$,
            \item \label{Thm:CharComp_iiib} $\lim_{n \to \infty} \sigma_X A(n^{-1}) = 0$,
            \item \label{Thm:CharComp_iiic} for an arbitrary sequence $\mathcal{R}_n$ of subsets of $\mathcal{R}$ such that $\mu(\mathcal{R}_n) < \infty$ and $\mathcal{R}_n \nearrow \mathcal{R}$, we have that $\lim_{n \to \infty} \sup_{f \in A} \lVert f \chi_{\mathcal{R} \setminus \mathcal{R}_n} \rVert_X = 0$.
        \end{enumerate}
    \end{enumerate} 
\end{theorem}

\begin{proof}
    As both topologies are metrisable, we only need to work with sequences. Thus, the sufficiency of \ref{Thm:CharComp_ii} for \ref{Thm:CharComp_i} is provided directly by Theorem~\ref{Thm:CharConv}. That \ref{Thm:CharComp_iii} implies \ref{Thm:CharComp_ii} is immediate. 
    
    As for the remaining implication, if we assume that \ref{Thm:CharComp_i} holds, then the condition \ref{Thm:CharComp_iii}\ref{Thm:CharComp_iiia} is satisfied by Theorem~\ref{Thm:EmbToMeasure}. Further, if we assume that \ref{Thm:CharComp_iii}\ref{Thm:CharComp_iiib} fails, then there is an $\varepsilon$ such that we may find for every $n \in \mathbb{N}$ some function $f_n \in A$ satisfying $\sigma_X f_n(n^{-1}) >\varepsilon$. This sequence $f_n$ cannot have any convergent subsequence in $(X, \lVert \cdot \rVert_X)$, as existence of such a subsequence would contradict Theorem~\ref{Thm:CharConv}. Similarly, if we assume that \ref{Thm:CharComp_iii}\ref{Thm:CharComp_iiib} fails, then there is a sequence $\mathcal{R}_n$ of subsets of $\mathcal{R}$ such that $\mu(\mathcal{R}_n) < \infty$ and $\mathcal{R}_n \nearrow \mathcal{R}$ and that we may find for every $n$ some function $f_n \in A$ satisfying $\lVert f \chi_{\mathcal{R} \setminus \mathcal{R}_n} \rVert_X > \varepsilon$. Just as before, this sequence $f_n$ cannot have any convergent subsequence in $(X, \lVert \cdot \rVert_X)$ because of Theorem~\ref{Thm:CharConv}.
\end{proof} 

The proof above is the source of the motivation for the stronger condition \ref{Thm:CharComp_iii}\ref{Thm:CharComp_iiic}. A failure of the weaker condition \ref{Thm:CharComp_ii}\ref{Thm:CharComp_iic} provides us with a counter-example function for every set of finite measure, but we need a sequence, so we must restrict ourselves to a countable family of sets. However, we then do not immediately see whether convergence of a subsequence violates Theorem~\ref{Thm:CharConv} \ref{Thm:CharConv_ii}\ref{Thm:CharConv_iic}; the sets we have used to construct the sequence are of course bad, but it is not obvious that there are no other sets of finite measure that would ensure $\delta_X A(n) \to 0$. The local condition $\sigma_X A (n^{-1}) \to 0$ must come into play and facilitate the pass from the hypothetical ``good'' sequence of sets to our counter-example one, just as at the end of the proof of Theorem~\ref{Thm:CharConv}.

However, it seems worth noting that the above-presented considerations may fail to be relevant for particular choices of the underlying measure spaces.

\begin{remark}
    
    Consider the measure space $(\mathbb{N}, m)$, where $m$ is the classical counting measure, and an arbitrary sequence of sets $\mathcal{R}_n$ with $m(\mathcal{R}_n) < \infty$ and $\mathcal{R}_n \nearrow \mathbb{N}$. Then every set $E \subseteq \mathbb{N}$ with $m(E) < \infty$ is a subset of $\mathcal{R}_n$ for $n$ large enough. Consequently, the conditions \ref{Thm:CharComp_iic} in parts \ref{Thm:CharComp_ii} and \ref{Thm:CharComp_iii} of Theorem~\ref{Thm:CharComp} are equivalent.
\end{remark}

Let us now consider the conditions $\sigma_X A(n^{-1}) \to 0$ and $\delta_X A(n) \to 0$ separately. Inspired by \cite[Theorems~4.3 and 4.6]{MihulaPandy25}, we present the following results. The first two of them show that if we appropriately strengthen the $X$-independent assumption $f_n \to f$ in $(\mathcal{M}_0, \mu_{\loc} )$, we may prove convergence using only one of the said conditions.

\begin{theorem} \label{Thm:ConvLoc}
    Let $\norm{\cdot}_X$ be a quasi-Banach function norm and let $X$ be the corresponding quasi-Banach function space. Let $f_k$ be a sequence of functions in $X$, $f \in X$, and denote $A = \{ f_k; \; k \in \mathbb{N} \} \cup \{ f \}$. Assume that
    \begin{enumerate}
        \item $\lim_{k \to \infty} \mu(\supp (f_k - f) ) = 0$,
        \item $\lim_{n \to \infty} \sigma_X A(n^{-1}) = 0$.
    \end{enumerate} 
    Then $f_k \to f$ in $(X, \lVert \cdot \rVert_X)$.
\end{theorem}

\begin{proof}
    Fix $\varepsilon >0$ and find in turn $n \in \mathbb{N}$ such that $\sigma_X A(n^{-1}) < \varepsilon$ and $k_0 \in \mathbb{N}$ such that $\mu(\supp (f_k - f) ) < n^{-1}$ for every $k \geq k_0$. Then we have for every such $k$
    \begin{equation*}
        \lVert f_k - f \rVert_X = \lVert (f_k - f) \chi_{\supp (f_k - f)} \rVert_X \lesssim \lVert f_k \chi_{\supp (f_k - f)} \rVert_X + \lVert f \chi_{\supp (f_k - f)} \rVert_X \lesssim \varepsilon.
    \end{equation*}
\end{proof} 

\begin{theorem} \label{Thm:ConvGlob}
    Let $\norm{\cdot}_X$ be a quasi-Banach function norm and let $X$ be the corresponding quasi-Banach function space. Let $f_k$ be a sequence of functions in $X$, $f \in X$, and denote $A = \{ f_k; \; k \in \mathbb{N} \} \cup \{ f \}$. Assume that
    \begin{enumerate}
        \item for every set $E$ with $\mu(E) < \infty$, we have that $f_k \chi_E \to f \chi_E$ in $(L^{\infty}, \lVert \cdot \rVert_{L^{\infty}})$,
        \item $\lim_{n \to \infty} \delta_X A(n) = 0$.
    \end{enumerate} 
    Then $f_k \to f$ in $(X, \lVert \cdot \rVert_X)$.
\end{theorem}

\begin{proof}
    Fix $\varepsilon > 0$ and find in turn some $n \in \mathbb{N}$ such that $\delta_X A(n) < \varepsilon$ and some set $F_0$ with $\mu(F_0) \leq n$ such that
    \begin{equation*}
        \lVert g \chi_{\mathcal{R} \setminus F_0} \rVert_X < \varepsilon
    \end{equation*} 
    for every $g \in A$. We proceed to find some $k_0 \in \mathbb{N} $ such that
    \begin{equation*}
        \lVert (f_k - f) \chi_{F_0} \rVert_{L^{\infty}} < \frac{\varepsilon}{\lVert \chi_{F_0} \rVert_X}
    \end{equation*}
    for every $k \geq k_0$. Then we have for every such $k$
    \begin{equation*}
    \begin{split}
        \lVert f_k - f \rVert_X \lesssim \lVert (f_k - f) \chi_{F_0} \rVert_X + \lVert f \chi_{\mathcal{R} \setminus F_0} \rVert_X + \lVert f_k \chi_{\mathcal{R} \setminus F_0}  \rVert_X
        \lesssim \varepsilon.
    \end{split}
    \end{equation*}
\end{proof}

For the sake of consistency and in preparation for another result, we also present a version of the last Theorem that works with a variant of the stronger global condition from the parts \ref{Thm:CharConv_iii} of Theorems~\ref{Thm:CharConv} and \ref{Thm:CharComp}.

\begin{theorem} \label{Thm:ConvGlobB}
    Let $\norm{\cdot}_X$ be a quasi-Banach function norm and let $X$ be the corresponding quasi-Banach function space. Let $f_k$ be a sequence of functions in $X$, $f \in X$, and denote $A = \{ f_k; \; k \in \mathbb{N} \} \cup \{ f \}$. Assume that there is a sequence $\mathcal{R}_n$ of subsets of $\mathcal{R}$ satisfying $\mu(\mathcal{R}_n) < \infty$ and $\mathcal{R}_n \nearrow \mathcal{R}$ such that the following two conditions hold:
    \begin{enumerate}
        \item for every $n \in \mathbb{N}$, we have that $f_k \chi_{\mathcal{R}_n} \to f \chi_{\mathcal{R}_n}$ in $(L^{\infty}, \lVert \cdot \rVert_{L^{\infty}})$,
        \item $\lim_{n \to \infty} \sup_{g \in A} \lVert g \chi_{\mathcal{R} \setminus \mathcal{R}_n} \rVert = 0$.
    \end{enumerate} 
    Then $f_k \to f$ in $(X, \lVert \cdot \rVert_X)$.
\end{theorem}

\begin{proof}
    Fix $\varepsilon > 0$ and find in turn some $n \in \mathbb{N}$ such that
    \begin{equation*}
        \lVert g \chi_{\mathcal{R} \setminus \mathcal{R}_{n}} \rVert_X < \varepsilon
    \end{equation*} 
    for every $g \in A$ and some $k_0 \in \mathbb{N} $ such that
    \begin{equation*}
        \lVert (f_k - f) \chi_{\mathcal{R}_{n}} \rVert_{L^{\infty}} < \frac{\varepsilon}{\lVert \chi_{\mathcal{R}_{n}} \rVert_X}
    \end{equation*}
    for every $k \geq k_0$. Then we have for every such $k$
    \begin{equation*}
    \begin{split}
        \lVert f_k - f \rVert_X \lesssim \lVert (f_k - f) \chi_{\mathcal{R}_{n}} \rVert_X + \lVert f \chi_{\mathcal{R} \setminus \mathcal{R}_{n}} \rVert_X + \lVert f_k \chi_{\mathcal{R} \setminus \mathcal{R}_{n}}  \rVert_X
        \lesssim \varepsilon.
    \end{split}
    \end{equation*}
\end{proof}

If we do not consider a fixed sequence, but rather a set $A$ with an appropriate property, then we may obtain the following characterisations.

\begin{theorem} \label{Thm:CharBoth}
    Let $\norm{\cdot}_X$ be a quasi-Banach function norm and let $X$ be the corresponding quasi-Banach function space. Let $A \subseteq X$ have the following property:
    \begin{equation} \label{Thm:CharBoth:Ideal}
        \text{If } f \in A \text{ and } g \leq f \text{ $\mu$-a.e.~then } g \in A.
    \end{equation}
    Then we have the following four characterisations:
    \begin{enumerate}
        \item  \label{Thm:CharBothLoc} The following two statements are equivalent:
        \begin{enumerate}
            \item $\lim_{n \to \infty} \sigma_X A(n^{-1}) = 0$.
            \item Every sequence $f_k$ of functions in $A$ that satisfies
            \begin{equation}  \label{Thm:CharBoth:LocCond}
                \lim_{k \to \infty} \mu(\supp f_k  ) = 0
            \end{equation}
            converges to $0$ in $(X, \lVert \cdot \rVert_X)$.
        \end{enumerate}
        \item  \label{Thm:CharBothGlob} The following two statements are equivalent:
        \begin{enumerate}
            \item $\lim_{n \to \infty} \delta_X A(n) = 0$.
            \item \label{Thm:CharBothGlobCond} Every net $f_{\iota}$ of functions in $A$ that satisfies 
            \begin{equation} \label{Thm:CharBoth:GlobCond}
                f_{\iota} \chi_E \to 0 \text{ in } (L^{\infty}, \lVert \cdot \rVert_{L^{\infty}})  \text{ for every } E \subseteq \mathcal{R} \text{ with } \mu(E) < \infty
            \end{equation}
            converges to $0$ in $(X, \lVert \cdot \rVert_X)$.
        \end{enumerate}
        \item  \label{Thm:CharBothGlobB} Let $\mathcal{R}_n$ be a sequence of subsets of $\mathcal{R}$ such that $\mu(\mathcal{R}_n) < \infty$ and $\mathcal{R}_n \nearrow \mathcal{R}$. Then the following two statements are equivalent:
        \begin{enumerate}
            \item $\lim_{n \to \infty} \sup_{f \in A} \lVert f \chi_{\mathcal{R} \setminus \mathcal{R}_n} \rVert_X = 0$.
            \item \label{Thm:CharBothGlobCondB} Every sequence $f_{k}$ of functions in $A$ that satisfies 
            \begin{equation} \label{Thm:CharBoth:GlobCondB}
                f_{k} \chi_{\mathcal{R}_n} \to 0 \text{ in } (L^{\infty}, \lVert \cdot \rVert_{L^{\infty}})  \text{ for every } n \in \mathbb{N}
            \end{equation}
            converges to $0$ in $(X, \lVert \cdot \rVert_X)$.
        \end{enumerate}
        \item  \label{Thm:CharBothGlobC} The following two statements are equivalent:
        \begin{enumerate}
            \item It holds for every sequence $\mathcal{R}_n$ of subsets of $\mathcal{R}$ such that $\mu(\mathcal{R}_n) < \infty$ and $\mathcal{R}_n \nearrow \mathcal{R}$ that $\lim_{n \to \infty} \sup_{f \in A} \lVert f \chi_{\mathcal{R} \setminus \mathcal{R}_n} \rVert_X = 0$.
            \item Every sequence $f_{k}$ of functions in $A$ for which there exists a sequence $\mathcal{R}_n$ of subsets of $\mathcal{R}$ satisfying $\mu(\mathcal{R}_n) < \infty$ and $\mathcal{R}_n \nearrow \mathcal{R}$ such that the sequence $f_k$ satisfies
            \begin{equation*} \label{Thm:CharBoth:GlobCondC}
                f_{k} \chi_{\mathcal{R}_n} \to 0 \text{ in } (L^{\infty}, \lVert \cdot \rVert_{L^{\infty}}),  \text{ for every } n \in \mathbb{N},
            \end{equation*}
            converges to $0$ in $(X, \lVert \cdot \rVert_X)$.
        \end{enumerate}
    \end{enumerate} 
\end{theorem}

The meaning and intention of the conditions \ref{Thm:CharBothGlob}\ref{Thm:CharBothGlobCond}, \ref{Thm:CharBothGlobB}\ref{Thm:CharBothGlobCondB}, and \ref{Thm:CharBothGlobC}\ref{Thm:CharBoth:GlobCondC} are probably not apparent; however, the discussion is not exactly brief, so we will present it after the proof.

\begin{proof}
    The sufficiency in all cases follows from Theorems~\ref{Thm:ConvLoc}, \ref{Thm:ConvGlob}, and \ref{Thm:ConvGlobB}, respectively; only in the case of Theorem~\ref{Thm:ConvGlob} one has to quickly realise that the argument works for nets as well (one literally only has to substitute $\iota$ instead of $k$).

    As for the necessity in \ref{Thm:CharBothLoc}, assume that
    \begin{equation*}
        \lim_{n \to \infty} \sigma_X A(n^{-1}) = 2\varepsilon > 0
    \end{equation*}
    and proceed to find for every $n \in \mathbb{N}$ some function $f_n \in A$ and some set $E_n \subseteq \mathcal{R}$ such that $\mu(E_n) < n^{-1}$ and $\lVert f_n \chi_{E_n} \rVert_X > \varepsilon$. Then the functions $f_n \chi_{E_n}$ belong to $A$ (as per \eqref{Thm:CharBoth:Ideal}), clearly satisfy \eqref{Thm:CharBoth:LocCond}, but they fail to converge to zero in $(X, \lVert \cdot \rVert_X)$.

    As for the necessity in \ref{Thm:CharBothGlob}, assume that
    \begin{equation*}
        \lim_{n \to \infty} \delta_X A(n) = \varepsilon > 0.
    \end{equation*}
    Thence we have for every set $E \subseteq \mathcal{R}$ with $\mu(E) < \infty$ some function $f_E \in A$ with
    \begin{equation*}
        \lVert f_E \chi_{\mathcal{R} \setminus E} \rVert_X > \varepsilon.
    \end{equation*}
    Considering the natural inclusion-order on subsets of $\mathcal{R}$ with finite measure, we observe that $g_E = f_E \chi_{\mathcal{R} \setminus E}$ is a net of functions in $A$ (as per \eqref{Thm:CharBoth:Ideal}). Furthermore, given a fixed set $F \subseteq \mathcal{R}$ with $\mu(F) < \infty$ we have for every $E \supseteq F$ that $g_E \chi_F = 0$, whence the validity of \eqref{Thm:CharBoth:GlobCond} is obvious. However, the net clearly fails to converge to zero in $(X, \lVert \cdot \rVert_X)$.

    As for the necessity in \ref{Thm:CharBothGlobB}, assume that
    \begin{equation*}
        \lim_{n \to \infty} \sup_{g \in A} \lVert g \chi_{\mathcal{R} \setminus \mathcal{R}_n} \rVert_X = \varepsilon > 0.
    \end{equation*}
    Thence we have for every $n \in \mathbb{N}$ some function $f_n \in A$ with
    \begin{equation*}
        \lVert f_n \chi_{\mathcal{R} \setminus \mathcal{R}_n} \rVert_X > \varepsilon.
    \end{equation*}
    The functions $f_n \chi_{\mathcal{R} \setminus \mathcal{R}_n}$ belong to $A$ (as per \eqref{Thm:CharBoth:Ideal}), clearly satisfy \eqref{Thm:CharBoth:GlobCondB}, but fail to converge to zero in $(X, \lVert \cdot \rVert_X)$.

    Finally, for the necessity in \ref{Thm:CharBothGlobC}, one only considers a sequence $\mathcal{R}_n$ for which 
    \begin{equation*}
        \lim_{n \to \infty} \sup_{g \in A} \lVert g \chi_{\mathcal{R} \setminus \mathcal{R}_n} \rVert_X = \varepsilon > 0.
    \end{equation*}
    Then the already established necessity in \ref{Thm:CharBothGlobB} applies.
\end{proof} 

Let us now provide some commentary. Firstly, we note that the condition \eqref{Thm:CharBoth:Ideal} is not very restrictive. Indeed, the validity of the conditions
\begin{enumerate}
    \item $\lim_{n \to \infty} \sigma_X A(n^{-1}) = 0$,
    \item $\lim_{n \to \infty} \delta_X A(n) = 0$,
    \item $\lim_{n \to \infty} \sup_{f \in A} \lVert f \chi_{\mathcal{R} \setminus \mathcal{R}_n} \rVert_X = 0$,
\end{enumerate}
for a given set $A$ is equivalent to the validity of the respective conditions for the set
\begin{equation*}
    \{ g \in \mathcal{M}_0 ; \; \exists f \in A, \, g \leq f\}.
\end{equation*}

Secondly, the rather unfortunate net-based characterisation in part \ref{Thm:CharBothGlob} is due to the fact that it is problematic to ensure that a sequence converges uniformly on every set of finite measure, yet that is what we need to combine with the condition $\delta_X A \to 0$ in order to obtain convergence in the quasinorm. The crux of the issue seems to be the fact that the space $L^{\infty}$ equipped with the locally convex topology induced by the family of seminorms
\begin{align*}
    f &\mapsto \esssup \, \lvert f \rvert \chi_E &\text{for every } E \subseteq \mathcal{R} \text{ with } \mu(E) < \infty 
\end{align*}
fails to be metrisable in general (e.g.~when the underlying measure space $(\mathcal{R}, \mu)$ is non-atomic and of infinite measure).

On the other hand, the condition:
\begin{equation*}
    \lim_{n \to \infty} \delta_X A(n) = 0
\end{equation*}
is equivalent to the statement:
\begin{itemize}[leftmargin=5mm]
    \item There exists a sequence $\mathcal{R}_n$ of subsets of $\mathcal{R}$ such that $\mu(\mathcal{R}_n) < \infty$ and $\mathcal{R}_n \nearrow \mathcal{R}$ for which 
    \begin{equation*}
        \lim_{n \to \infty} \sup_{f \in A} \lVert f \chi_{\mathcal{R} \setminus \mathcal{R}_n} \rVert_X = 0.
    \end{equation*}
\end{itemize}
We can thus apply part \ref{Thm:CharBothGlobB} on this specific sequence of sets. However, the resulting characterising condition would then read:
\begin{itemize}[leftmargin=5mm]
    \item There exists a sequence $\mathcal{R}_n$ of subsets of $\mathcal{R}$ such that $\mu(\mathcal{R}_n) < \infty$ and $\mathcal{R}_n \nearrow \mathcal{R}$ for which it holds that every sequence $f_k$ of functions in $A$ that satisfies
    \begin{equation*}
                f_{k} \chi_{\mathcal{R}_n} \to 0 \text{ in } (L^{\infty}, \lVert \cdot \rVert_{L^{\infty}})  \text{ for every } n \in \mathbb{N}
    \end{equation*}
    converges to $0$ in $(X, \lVert \cdot \rVert_X)$.
\end{itemize}
This does not seem to be an improvement over the net-based condition in part \ref{Thm:CharBothGlob}.

\subsection{An alternative approach} \label{SecAmalgamTypeAlternative}

Next, we present an alternative approach to the problem at hand which yields rather interesting results that could not have been easily formulated without the amalgam approach and thus have no classical analogue. The principal idea is that the condition 
\begin{equation*}
    \lim_{n \to \infty} \delta_X A(n) = 0
\end{equation*}
can be weakened to just
\begin{equation*}
    \lim_{n \to \infty} \sup_{g \in A} \delta_X g(n) = 0,
\end{equation*}
at the cost of strengthening the $X$-independent assumption $f_k \to f$ in $(\mathcal{M}_0, \mu_{\loc} )$ to full convergence in measure and adding some further assumptions on $X$ and $(\mathcal{R}, \mu)$. Unfortunately, these necessary additional assumptions seem to be distinct for the individual implications, whence we present said implications in separate statements. Furthermore, as the topology that induces the full convergence in measure is rather less well-behaved than $\mu_{\loc}$ (e.g.~scalar multiplication fails to be continuous on the entirety of $\mathcal{M}_0$, see e.g.~\cite[Chapter~IV, Section~11]{DunfordSchwartz56}), we only present the case of convergence.

\begin{theorem}\label{Thm:ConvAlt}
    Let $(\mathcal{R}, \mu)$ be resonant, let $\norm{\cdot}_X$ be a quasi-Banach function norm, let $X$ be the corresponding quasi-Banach function space, and assume that $\varphi_X^{\max}(t) < \infty$ for all $t \in [0, \infty)$. Let $f_k$ be a sequence of functions in $X$, $f \in X$, and denote $A = \{ f_k; \; k \in \mathbb{N} \} \cup \{ f \}$. Assume that
    \begin{enumerate}
        \item $f_k \to f$ in measure,
        \item $\lim_{n \to \infty} \sup_{g \in A} \sigma_X g(n^{-1}) = 0$,
        \item $\lim_{n \to \infty} \sup_{g \in A} \delta_X g(n) = 0$.
    \end{enumerate} 
    Then $f_k \to f$ in $(X, \lVert \cdot \rVert_X)$.
\end{theorem}

Recall that the assumption that $\varphi_X^{\max}$ is finite is always satisfied when $(\mathcal{R}, \mu)$ is non-atomic (Theorem~\ref{lenka4.4}).

\begin{proof}
    Fix $\varepsilon > 0$ and find $n \in \mathbb{N}$ such that $\sup_{g \in A} \delta_X g(n) < \varepsilon$ and $\sup_{g \in A} \sigma_X g(n^{-1}) < \varepsilon$. Find further for any $g \in A$, some set $F_g$ with $\mu(F_g) \leq n$ such that
    \begin{equation*}
        \lVert g \chi_{\mathcal{R} \setminus F_g} \rVert_X < \varepsilon
    \end{equation*} 
    and denote $F_k = F_f \cup F_{f_k}$; we have $\mu(F_k) \leq 2n$. Finally, put
    \begin{equation*}
        E_k := \left \{ x \in \mathcal{R}; \; \lvert f_k - f \rvert > \frac{\varepsilon}{\varphi_X^{\max}(2n)} \right \},
    \end{equation*}
    where the finiteness of $\varphi_X^{\max}(2n)$ is included in our assumptions, and use the convergence in measure to find some $k_0 \in \mathbb{N}$ such that for every $k \geq k_0$ we have $\mu (E_k) < n^{-1}$. Then we have for any such $k$
    \begin{equation*}
    \begin{split}
        \lVert f_k - f \rVert_X &\lesssim \lVert f \chi_{E_k} \rVert_X + \lVert f_k \chi_{E_k} \rVert_X + \lVert (f_k - f) \chi_{F_k \setminus E_k} \rVert_X + \lVert f \chi_{\mathcal{R} \setminus F_k} \rVert_X + \lVert f_k \chi_{\mathcal{R} \setminus F_k} \rVert_X \\
        &< \varepsilon \left( 4 + \frac{\lVert \chi_{F_k} \rVert_X}{\varphi_X^{\max}(2n)} \right) \\
        &\lesssim \varepsilon.
    \end{split}
    \end{equation*}
    The last step employs the assumption that $(\mathcal{R}, \mu)$ is resonant (see Remark~\ref{Rem:FundMono}).
\end{proof} 

As for the converse implication, the $X$-dependent conditions are weaker than in Theorem~\ref{Thm:CharConv}, so we only have to consider the convergence in measure. Note that we find it somehow surprising that this result holds without assuming that $X$ is r.i.

\begin{theorem}\label{Thm:measureConvChar}
    Let $\norm{\cdot}_X$ be a quasi-Banach function norm and let $X$ be the corresponding quasi-Banach function space. Consider the following two statements:
    \begin{enumerate}
        \item \label{Thm:measureConvChar_i} $\norm{\cdot}_X$ is admissible,
        \item \label{Thm:measureConvChar_ii} for every $f_k,f\in X$ such that $f_k\to f$ in $(X, \lVert \cdot \rVert_X)$, it holds that $f_k\to f$ in measure.
    \end{enumerate}
    Then \ref{Thm:measureConvChar_ii} implies \ref{Thm:measureConvChar_i}. Furthermore, if $(\mathcal{R}, \mu)$ is resonant, then also \ref{Thm:measureConvChar_i} implies \ref{Thm:measureConvChar_ii}.
\end{theorem}
\begin{proof}
    We will prove both implications by contraposition. To show that \ref{Thm:measureConvChar_ii} implies \ref{Thm:measureConvChar_i} suppose that $\norm{\cdot}_X$ is not admissible. Hence, there exists some $t>0$ in the range of $\mu$ and a sequence of sets $E_k\subseteq \mathcal{R}$, $\mu(E_k)=t$ such that 
    \begin{equation*}
        \lim_{k\to\infty} \norm{\chi_{E_k}}_X = 0,
    \end{equation*}
    which means that $\chi_{E_k}\to 0$ in $(X, \lVert \cdot \rVert_X)$. However, we have that
    \begin{equation*}
        \mu\left(\left\{ x\in \R ; \; \chi_{E_k}\geq 1\right\}\right)=t,
    \end{equation*}
    and so $\chi_{E_k}$ does not converge to $0$ in measure. 

    For the other implication, suppose that there exist $f_k,f\in X$ such that $f_k\to f$ in $(X, \lVert \cdot \rVert_X)$ and $f_k\not\to f$ in measure. Then there exists $\varepsilon>0$ and a subsequence $f_{k_j}$ of $f_k$ such that the sets
    \begin{equation*}
        E_j:=\left\{ x\in \R ; \; \abs{f-f_{k_j}} > \varepsilon \right\}
    \end{equation*}
    satisfy
    \begin{equation*}
        \lim_{j\to\infty} \mu\left( E_{j} \right) > 0.
    \end{equation*}
    The assumption that $(\mathcal{R}, \mu)$ is resonant implies that there exists $t>0$ in the range of $\mu$ and $j_0\in\N$ such that $\mu\left(E_{j}\right) \geq t$ for all $j\geq j_0$. Hence, we may find for every $j \geq j_0$ some subset $\widetilde{E}_j \subseteq E_j$ satisfying $\mu ( \widetilde{E}_j ) = t$. Then
    \begin{equation*}
        0 \leq \varphi_X^{\min}(t) \leq \lVert \chi_{\widetilde{E}_j} \rVert_X \leq \frac{1}{\varepsilon} \lVert (f - f_{k_j}) \chi_{\widetilde{E}_j} \rVert_X \leq \frac{1}{\varepsilon} \lVert f - f_{k_j}\rVert_X \to 0.
    \end{equation*}
\end{proof}

Similarly as in the previous section, we also consider the condition $\sup_{g \in A} \delta_X g(n) \to 0$ separately.

\begin{theorem}
    Let $(\mathcal{R}, \mu)$ be resonant, let $\norm{\cdot}_X$ be a quasi-Banach function norm, let $X$ be the corresponding quasi-Banach function space, and assume that $\varphi_X^{\max}(t) < \infty$ for all $t \in [0, \infty)$. Let $f_k$ be a sequence of functions in $X$, $f \in X$, and denote $A = \{ f_k; \; k \in \mathbb{N} \} \cup \{ f \}$. Assume that
    \begin{enumerate}
        \item $f_k \to f$ in $(L^{\infty}, \lVert \cdot \rVert_{L^{\infty}})$,
        \item $\lim_{n \to \infty} \sup_{g \in A} \delta_X g(n) = 0$.
    \end{enumerate} 
    Then $f_k \to f$ in $(X, \lVert \cdot \rVert_X)$.
\end{theorem}

\begin{proof}
    Fix $\varepsilon > 0$ and find $n \in \mathbb{N}$ such that $\sup_{g \in A} \delta_X g(n) < \varepsilon$. Find further for any $g \in A$ some set $F_g$ with $\mu(F_g) \leq n$ such that
    \begin{equation*}
        \lVert g \chi_{\mathcal{R} \setminus F_g} \rVert_X < \varepsilon
    \end{equation*} 
    and denote $F_k = F_f \cup F_{f_k}$; we have $\mu(F_k) \leq 2n$. We proceed to find some $k_0 \in \mathbb{N} $ such that
    \begin{equation*}
        \lVert f_k - f \rVert_{L^{\infty}} < \frac{\varepsilon}{\varphi_X^{\max}(2n)}
    \end{equation*}
    for every $k \geq k_0$; note that the finiteness of $\varphi_X^{\max}(2n)$ is part of our assumptions. Then we have for every such $k$
    \begin{equation*}
    \begin{split}
         \lVert f_k - f \rVert_X &\lesssim \lVert (f_k - f) \chi_{F_k} \rVert_X + \lVert f \chi_{\mathcal{R} \setminus F_k} \rVert_X + \lVert f_k \chi_{\mathcal{R} \setminus F_k}  \rVert_X \\
         &\leq \varepsilon \left (2 + \frac{\lVert \chi_{F_k} \rVert_X}{\varphi_X^{\max}(2n)}  \right ) \\
         &\lesssim \varepsilon.
    \end{split}
    \end{equation*}
    The last step employs the assumption that $(\mathcal{R}, \mu)$ is resonant (see Remark~\ref{Rem:FundMono}).
\end{proof} 

As for the converse implication, the situation is different from that in the previous section. In this case we were only able to obtain it by adding additional a~priori assumptions on $A$ (namely, $\lVert \cdot \rVert_X$-boundedness) and $X$ (namely, $\lim_{t \to \infty} \varphi_X^{\min}(t) = \infty$). Given that those assumptions do not seem natural and that they turn the result into a close copy of \cite[Theorem~4.3]{MihulaPandy25}, we will be content to refer to said paper and not present the result here.

\subsection{Uniform absolute continuity of the quasinorm and almost-compact embeddings} \label{SecAmalgamTypeUACqN}
 
We now present the amalgam-type characterisation of the uniform absolute continuity of the quasinorm. Although the proof is virtually identical to that of Theorem~\ref{Thm:ac_char}, we present it for the sake of completeness and because the important characterisation of almost-compact embeddings is a direct consequence of this result.

\begin{theorem}\label{Thm:uac_char}
    Let $\norm{\cdot}_X$ be a quasi-Banach function norm and let $X$ be the corresponding quasi-Banach function space. Let $A \subseteq X$. Then $A$ has uniformly absolutely continuous quasinorm if and only if the following two conditions hold:
    \begin{equation}\label{Thm:uac_char:e_loc}
        \lim_{n\rightarrow\infty} \sigma_X A\left(n^{-1}\right) = 0,
    \end{equation}
    and
    \begin{equation}\label{Thm:uac_char:e_glob}
        \lim_{n\rightarrow\infty} \delta_X A(n) = 0.
    \end{equation}
\end{theorem}

\begin{proof}
    We start with the necessity for~\eqref{Thm:uac_char:e_loc}. Let $\varepsilon>0$. 
    For every $n\in\N$, we find a set $E_n\subseteq \mathcal{R}$ such that $\mu(E_n)\leq n^{-1}$ and 
    \begin{equation*}
        \sigma_X A\left(n^{-1}\right)  -\varepsilon < \sup_{f \in A} \norm{f \chi_{E_n}}_X.
    \end{equation*}
    We have $\mu(E_n)\rightarrow 0$. Hence, there exists a subsequence $E_{n_{k}}$ such that $\chi_{E_{n_{k}}}\rightarrow 0$ a.e. Because $A$ has uniformly absolutely continuous quasinorm, we get that there exists $k_0\in\N$ such that we have for every $k\geq k_0$ and every $f \in A$ that
    \begin{equation*}
        \norm{f \chi_{E_{n_{k}}}}_X < \varepsilon.
    \end{equation*}
    Thus, we obtain for every $k\geq k_0$ that
    \begin{equation*}
        \sigma_X A\left(n_k^{-1}\right) < \sup_{f \in A} \norm{f \chi_{E_{n_k}}}_X + \varepsilon < 2 \varepsilon,
    \end{equation*}
    which is enough due to the monotonicity of our sequence. 
    
    For the necessity of~\eqref{Thm:uac_char:e_glob}, we use the $\sigma$-finiteness of $(\mathcal{R},\mu)$ to find a sequence of sets $E_n\subseteq \mathcal{R}$ such that $E_n \nearrow \mathcal{R}$ and $\mu(E_n)<n$ for every $n\in\N$. Then $\chi_{\mathcal{R}\setminus E_n}\rightarrow 0$ a.e. Let $\varepsilon >0$. Because $A$ has uniformly absolutely continuous quasinorm, we find $n_0\in\N$ such that for every $n\geq n_0$ we have
    \begin{equation*}
        \sup_{f \in A} \norm{f \chi_{\mathcal{R}\setminus E_n}}_X<\varepsilon.
    \end{equation*}
    Thence we obtain that for every $n\geq n_0$
    \begin{equation*}
        \delta_X A(n) \leq \sup_{f \in A} \norm{f \chi_{\mathcal{R}\setminus E_n}}_X < \varepsilon.
    \end{equation*} 
    
    For sufficiency, let $\varepsilon>0$ and let $\chi_{E_n}\rightarrow 0$ a.e.~for some $E_n \subseteq \mathcal{R}$. By~\eqref{Thm:uac_char:e_glob}, 
    we find $n_0\in\N$ such that 
    \begin{equation*}
        \delta_X A(n_0) < \varepsilon.
    \end{equation*}
    Hence, there exists a set $E_0$ such that $\mu(E_0)<n_0$ and 
    \begin{equation}\label{E0_unif}
         \sup_{f \in A} \norm{f \chi_{\mathcal{R}\setminus E_0}}_X<\varepsilon.
    \end{equation}
    We denote $\widetilde{E}_n = E_n \cap E_0$ and note that $\mu( \widetilde{E}_n ) \rightarrow 0$ by the dominated convergence theorem. Next, by~\eqref{Thm:uac_char:e_loc}, we find $n_1\in\N$  
    such that for every set $E\subseteq \mathcal{R}$, $\mu(E)< n_1^{-1}$, we have 
    \begin{equation*}
         \sup_{f \in A} \norm{f \chi_{E}}_X<\varepsilon.
    \end{equation*}
    Because $\mu( \widetilde{E}_n ) \rightarrow 0$, we find $n_2\in\N$, $n_2 \geq n_0$, such that for every $n\geq n_2$ we have $\mu( \widetilde{E}_n ) < n_1^{-1}$ and by 
    the previous equation
    \begin{equation}\label{EK_unif}
        \sup_{f \in A} \norm{f \chi_{\widetilde{E}_n}}_X<\varepsilon,
    \end{equation}
    for every $n\geq n_2$. Using~\eqref{E0_unif}, \eqref{EK_unif}, and the subadditivity of suprema, we obtain for every $n\geq n_2$, that
    \begin{equation*}
        \sup_{f \in A} \norm{f \chi_{E_n}}_X \lesssim \sup_{f \in A} \norm{f\chi_{\widetilde{E}_k}}_X+\sup_{f \in A} \norm{f\chi_{\mathcal{R}\setminus E_0}}_X < 2 \varepsilon.
    \end{equation*}
\end{proof}

To get an amalgam-type characterisation of almost-compact embeddings one simply applies the preceding result onto $B_Y$. Let us note that the two conditions together easily imply that $Y \hookrightarrow X$.

\begin{corollary} \label{Cor:Char_ACE}
    Let $\norm{\cdot}_X$ and $\norm{\cdot}_Y$ be a quasi-Banach function norms over the same measure space $(\mathcal{R}, \mu)$ and let $X$ and $Y$ be the corresponding quasi-Banach function spaces. Then $Y \stackrel{*}{\hookrightarrow} X$ if and only if both    
    \begin{equation*}
        \lim_{n\rightarrow\infty} \sigma_X B_Y \left(n^{-1}\right) = 0,
    \end{equation*}
    and
    \begin{equation} \label{Cor:Char_ACE:e_glob}
        \lim_{n\rightarrow\infty} \delta_X B_Y(n) = 0.
    \end{equation} 
\end{corollary} 

The condition \eqref{Cor:Char_ACE:e_glob} is difficult to satisfy. This has been implicitly observed in \cite[Theorem~4.5]{Slavikova12}, where it was shown that there are no almost-compact embeddings among Banach function spaces over non-atomic measure spaces $(\mathcal{R}, \mu)$ with $\mu(\mathcal{R}) = \infty$, with the argument essentially showing the failure of \eqref{Cor:Char_ACE:e_glob}. As we are working in a wider setting, we now present a somewhat more general and more precise version of the result, although the argument is essentially the same.

\begin{proposition} \label{Prop:Global_fail}
    Let $\norm{\cdot}_X$ and $\norm{\cdot}_Y$ be a quasi-Banach function norms over the same measure space $(\mathcal{R}, \mu)$ and let $X$ and $Y$ be the corresponding quasi-Banach function spaces. Assume that there is some $t \in (0, \infty)$ such that the following three conditions hold:
    \begin{enumerate}
        \item \label{Prop:Global_fail_i} For every $E \subseteq \mathcal{R}$ with $\mu(E) < \infty$ there is some set $F_E \subseteq \mathcal{R} \setminus E$ with $\mu(F_E) = t$.
        \item \label{Prop:Global_fail_ii} We have $\varphi_Y^{\max} (t) < \infty$.
        \item \label{Prop:Global_fail_iii} We have $\varphi_X^{\min} (t) > 0$.
    \end{enumerate} 
    Then \eqref{Cor:Char_ACE:e_glob} fails to hold and consequently $Y \not \stackrel{*}{\hookrightarrow} X$.
\end{proposition} 

The condition \ref{Prop:Global_fail_i} of course implies that $(\mathcal{R}, \mu)$ is of infinite measure and is automatically satisfied when it is also resonant. The condition \ref{Prop:Global_fail_ii} is automatically satisfied when $(\mathcal{R}, \mu)$ is non-atomic (Theorem~\ref{lenka4.4}) or when $Y$ is r.i. The condition \ref{Prop:Global_fail_iii} is satisfied if $X$ is admissible, e.g.~when both $(\mathcal{R}, \mu)$ is non-atomic and $X$ is a Banach function space or when $X$ is r.i.

\begin{proof}
    Fix $\varepsilon > 0$. We find for every $n \in \mathbb{N}$ some set $E_n$ of measure $\mu(E_n) \leq n$ such that
    \begin{equation*}
        \delta_X B_Y (n) + \varepsilon > \sup_{f \in B_Y} \lVert f \chi_{\mathcal{R} \setminus E_n} \rVert_X.
    \end{equation*}
    Then we use \ref{Prop:Global_fail_i} to find the corresponding set $F_{E_n}$ and observe that \ref{Prop:Global_fail_ii} implies
    \begin{equation*}
        \frac{1}{\varphi_Y^{\max} (t)} \chi_{F_{E_n}}  \in B_Y.
    \end{equation*}
    Thus, we may apply  \ref{Prop:Global_fail_iii} to get
    \begin{equation*}
        \delta_X B_Y (n) + \varepsilon > \sup_{f \in B_Y} \lVert f \chi_{\mathcal{R} \setminus E_n} \rVert_X \geq \frac{1}{\varphi_Y^{\max} (t)} \lVert \chi_{F_{E_n}} \rVert_X \geq \frac{\varphi_X^{\min} (t)}{\varphi_Y^{\max} (t)} > 0.
    \end{equation*}
    As $\varepsilon$ and $n$ were arbitrary, we obtain the desired conclusion.
\end{proof}

It is apparent from the proof that we do not actually need the full estimates from \ref{Prop:Global_fail_ii} and \ref{Prop:Global_fail_iii}. Indeed, the proof would hold even if we had them for the specific sequence $F_{E_n}$ rather than all sets of measure $t$; all we need is the upper estimate on $\lVert \chi_{F_{E_n}} \rVert_Y$ and the lower estimate on $\lVert \chi_{F_{E_n}} \rVert_X$.

Let us now show that, unlike in the case or r.i.~spaces over resonant measure space (see \cite[Theorems~4.3 and 4.5]{Slavikova12}), there are some pairs of even Banach function spaces over a measure space of infinite measure for which we have that the embedding is almost-compact but not compact. We of course recognize that the example is rather artificial.

\begin{example}
    Consider the set $\mathcal{R} = [0, 1] \cup \mathbb{N}$ equipped with the measure $\mu$ that acts like the Lebesgue measure $\lambda$ on $[0,1]$ and like the counting measure $m$ on $\mathbb{N}$. Then the space $X$ given by the norm
    \begin{equation*}
        \lVert f \rVert_X = \esssup_{t \in [0,1]} \lvert f(t) \rvert + \sup_{n \in \mathbb{N}} n \lvert f(n) \rvert
    \end{equation*}
    satisfies $X \stackrel{*}{\hookrightarrow} L^{2}$, but the embedding is not compact.
\end{example}

In the light of the high failure rate of \eqref{Cor:Char_ACE:e_glob} and considering Theorem~\ref{Thm:ConvAlt}, it seems natural to look towards the weaker relation
\begin{equation*}
    \lim_{n \to \infty} \sup_{f \in B_Y} \delta_X f(n) = 0
\end{equation*}
as a possible replacement. This relation, together with the local counterpart
\begin{equation*}
    \lim_{n \to \infty} \sup_{f \in B_Y} \sigma_X f(n^{-1}) = 0,
\end{equation*}
has been thoroughly examined in \cite{MihulaPandy25}, so we will just refer to said paper and repeat that it was a great inspiration for our work. Let us just mention that these relations lead to some rather interesting conclusions and are also rather well motivated, as they were originally introduced in order to provide a characterisation of compact embeddings of Sobolev spaces of radially symmetric functions on the entire Euclidean space in \cite{Mihula26}.

\section{The case of rearrangement-invariant spaces} \label{SectionRI}

Let us now consider quasi-Banach function spaces that are r.i. In this setting, most of our conditions can be equivalently restated in the terms of the non-increasing rearrangement and the representation space in a rather natural way; the one significant exception is the condition $\delta_X A (t) \to 0$ (we shall discuss this in more detail after the proofs).

Throughout this section, we will only work with resonant measure spaces $(\mathcal{R}, \mu)$ and r.i.~quasi-Banach functions spaces $X$. To prevent any confusion, we stress that we always assume a~priori that $f \in \M_0(\mathcal{R},\mu)$, i.e.~it cannot happen that we would work with a function $g \in \mathcal{M}([0, \infty), \lambda)$ for which there would be no equimeasurable function in $\M_0(\mathcal{R},\mu)$. For an analogous reason, we will also prove the next two results only for $t$ in the range of $\mu$, although half of the estimates does not require this assumption.

We shall start with the simpler case, that is representing the maximal local norm of a function.

\begin{proposition}\label{Prop:local_ri_equiv}
    Let $(\mathcal{R}, \mu)$ be resonant, let $\norm{\cdot}_X$ be an r.i.~quasi-Banach function norm, and let $X$ be the corresponding r.i.~quasi-Banach function space. Let $t\geq0$ be in the range of $\mu$. 
    Then, for all $f\in \M_0(\mathcal{R},\mu)$, we have
    \begin{equation*}
        \sigma_X f(t)=\norm{f^*\chi_{[0,t)}}_{\overline{X}}.
    \end{equation*}
\end{proposition}

As this (and the next) result is similar in spirit to \cite[Theorem~4.8]{MihulaPandy25}, we believe it is worth pointing out the main conceptual difference: In our case, the representation is localised to an individual function $f \in \M_0(\mathcal{R},\mu)$, and we work with point-wise values rather than with limits at $0$ or $\infty$.

\begin{proof}
    Let $f\in \M_0(\mathcal{R},\mu)$. For the first inequality, let $\varepsilon>0$ and by the definition of $\sigma_X f$ find a set $G\subseteq \mathcal{R}$, $\mu(G)=t$ such that 
    \begin{equation*}
        \sigma_X f(t) < \norm{f\chi_G}_X + \varepsilon.
    \end{equation*}
    Then using the properties of the non-increasing rearrangement and the fact that $f\chi_G$ and $(f\chi_G)^*$ are equimeasurable, with the latter function being supported inside $[0, t)$, we obtain
    \begin{equation*}
        \sigma_X f(t) < \norm{f\chi_G}_X + \varepsilon = \norm{(f\chi_G)^* \chi_{[0,t)}}_{\overline{X}} + \varepsilon \leq \norm{f^* \chi_{[0,t)}}_{\overline{X}} + \varepsilon,
    \end{equation*}
    which gives
    \begin{equation*}
        \sigma_X f(t)\leq\norm{f^*\chi_{[0,t)}}_{\overline{X}}.
    \end{equation*}
    
    For the second inequality, we start by only considering $f\in \M_0(\mathcal{R} ,\mu)$ non-negative and with support of finite measure. We denote $E:=\supp f$. Note that the statement is trivial if $\mu(E)\leq t$, so we may assume that $\mu(E)>t$. We now need to distinguish two cases:
    \begin{itemize}[leftmargin=5mm]
        \item If the measure space $(\R,\mu)$ is non-atomic, then by \cite[Chapter~2, Theorem~7.5]{BennettSharpley88} there exists a measure preserving transformation $\sigma\colon E \rightarrow (0,\mu(E))$ such that $f=f^*\circ \sigma$ a.e.~on $E$. Next, we put $H:=\sigma^{-1}((0,t))$. Then $\mu(H)=t$~and $(f^*\chi_{[0,t)}) \circ \sigma = f \chi_H$ a.e.~on $E$. Hence, $f^*\chi_{[0,t)}$ and $f \chi_H$ are equimeasurable.
        \item If $(\R,\mu)$ is completely atomic with atoms of equal measure $\alpha$, then $f^*$ is a right-continuous decreasing step function with jumps occurring only at points $k\alpha$, $k\in\N$. In other words, it is constant on the intervals $[k\alpha, (k+1)\alpha)$, $k \in \mathbb{N}$. For every $k\in\N$ such that $(k+1)\alpha\leq t$ we find a distinct $x_k\in R$ satisfying $f(x_k)=f^*(k\alpha)$; this is possible because we assume that $\mu(E) < \infty$. Let $H\subseteq R$ be the collection of all $x_k$ such that $(k+1)\alpha\leq t$. From this construction, it is clear that the functions $f^*\chi_{[0,t)}$ and $f\chi_H$ are equimeasurable.
    \end{itemize}
    In both cases, this yields
    \begin{equation*}
        \norm{f^*\chi_{[0,t)}}_{\overline{X}} = \norm{f \chi_H}_X \leq \sigma_X f(t).
    \end{equation*}
    For a general non-negative $f\in\M_0(\mathcal{R},\mu)$, we let $f_n:=f\chi_{\mathcal{R}_n}$ for every $n\in\N$, where $\mathcal{R}_n\nearrow \mathcal{R}$ is the sequence from the $\sigma$-finiteness of $(\mathcal{R},\mu)$. The previously proven case gives
    \begin{equation*}
        \sigma_X f_n(t)=\norm{f^*_n\chi_{[0,t)}}_{\overline{X}},
    \end{equation*}
    for every $n\in\N$. We may now take the limit on both sides. Indeed, on the right-hand side, the convergence follows by using the properties of the non-increasing rearrangement and the property \ref{P3} of $\norm{\cdot}_{\overline{X}}$, while on the left-hand side one only has to interchange suprema and apply the property \ref{P3} of $\norm{\cdot}_{X}$. Finally, the extension from non-negative functions to general ones is trivial.
\end{proof}

The answer is not so simple for the global part, where we were only able to obtain the following inequalities with a substantially more involved proof. Note also that we assume that $\mu(\mathcal{R}) = \infty$; this is a technical assumption that makes the proof less complicated and that is not restrictive for our purposes as we are only interested in the convergence to zero as $t \to \infty$ (which is trivially satisfied for both functionals if $\mu(\mathcal{R}) < \infty$).

\begin{proposition}\label{global_norm_rearrangement}
   Let $(\mathcal{R}, \mu)$ be resonant, let $\norm{\cdot}_X$ be an r.i.~quasi-Banach function norm, and let $X$ be the corresponding r.i.~quasi-Banach function space. Assume that $\mu(\mathcal{R}) = \infty$. Let $t\geq0$ be in the range of $\mu$. Then it holds for all $f\in \M_0(\mathcal{R},\mu)$ that
    \begin{equation*}
        \delta_Xf(t) \lesssim \norm{f^*\chi_{[t,\infty)}}_{\overline{X}} \lesssim \delta_Xf\left(\frac{t}{2}\right),
    \end{equation*}
    where the hidden constants depend only on the space $X$, not on $f$ or $t$.
\end{proposition}

\begin{proof}
    We start with the second inequality. Let $\varepsilon>0$ and $f\in\M_0(\mathcal{R},\mu)$. By the definition of $\delta_X$, we find a set $H\subseteq \mathcal{R}$, $\mu(H) \leq t/2$, such that
    \begin{equation*}
        \norm{f\chi_{\mathcal{R}\setminus H}}_X < \delta_Xf\left(\frac{t}{2}\right) + \varepsilon.
    \end{equation*}
    From the properties of the non-increasing rearrangement, we get the following pointwise inequality
    \begin{equation*}
        g^*(s)-h^*\left(\frac{s}{2}\right)\leq (g-h)^*\left(\frac{s}{2}\right),
    \end{equation*}
    for all $g,h\in\M_0(\mathcal{R},\mu)$ and all $s \in (0,\infty)$. 
    We notice that $(f\chi_H)^*\left(s/2\right)=0$ for all $s\in [t,\infty)$,
    so for all $s\in [t,\infty)$, we get
    \begin{equation*}
        f^*(s)=f^*(s)-(f\chi_H)^*\left(\frac{s}{2}\right)\leq(f\chi_{\mathcal{R}\setminus H})^*\left(\frac{s}{2}\right).
    \end{equation*}
    Using this inequality, the property \ref{P2} of $\norm{\cdot}_{\overline{X}}$, and the boundedness of the dilation operator $D_{\frac{1}{2}} \colon \overline{X} \to \overline{X}$ (\cite[Theorem~3.23]{NekvindaPesa24}), we calculate
    \begin{equation*}
        \begin{aligned}
            \norm{f^*\chi_{[t,\infty)}}_{\overline{X}}
                        &\leq \norm{ D_{\frac{1}{2}} \left(\left(f\chi_{\mathcal{R}\setminus H}\right)^* \right)\chi_{[t,\infty)}}_{\overline{X}} \\
                        &\leq \norm{D_{\frac{1}{2}}}_{\overline{X}\rightarrow \overline{X}} \norm{\left(f\chi_{\mathcal{R}\setminus H}\right)^*}_{\overline{X}} \\
                        &=\norm{D_{\frac{1}{2}}}_{\overline{X}\rightarrow \overline{X}} \norm{f\chi_{\mathcal{R}\setminus H}}_{X} \\
                        &< \norm{D_{\frac{1}{2}}}_{\overline{X}\rightarrow \overline{X}} \left( \delta_Xf\left(\frac{t}{2}\right) + \varepsilon\right).
        \end{aligned}
    \end{equation*}
    
    We now move to the first inequality. Without loss of generality, we may consider only functions $f\in \M_+(\mathcal{R},\mu)\cap \M_0(\mathcal{R},\mu)$ and we denote
    \begin{equation*}
        \alpha:=\lim_{t\to\infty} f^*(t).
    \end{equation*}
    If $\mu(\supp f)\leq t$, the statement is clear because both terms evaluate to 0. Hence, we may only consider $\mu(\supp f)>t$. 
    
    First, suppose that $\alpha=0$. We need to distinguish two cases:
    \begin{itemize}[leftmargin=5mm]
        \item If the measure space $(\R,\mu)$ is non-atomic, then by \cite[Chapter~2, Corollary~7.6]{BennettSharpley88}, there exists a measure preserving transformation $\sigma\colon \supp f\rightarrow \supp f^*$ such that $f=f^*\circ \sigma$ a.e.~on $\supp f$. We put $G:=\sigma^{-1}((0,t))$. Then $\mu(G)=t$ and $f\chi_{\mathcal{R}\setminus G}=f^*\chi_{[t,\infty)}\circ \sigma$ a.e.~on $\supp f$. Thence, $f^*\chi_{[t,\infty)}$ and $f\chi_{\mathcal{R}\setminus G}$ are equimeasurable.
        \item If $(R,\mu)$ is completely atomic with atoms of equal measure $\beta$, then $f^*$ is a right-continuous decreasing step function with jumps occurring only at points $k\beta$, $k\in\N$. In other words, it is constant on the intervals $[k\beta, (k+1)\beta)$, $k \in \mathbb{N}$. Since $\alpha=0$, it is easy to construct an ordering of atoms $\sigma\colon \N \to \R$ such that $f(\sigma(k))=f^*(k\beta)$, $k\in\N$. Let $G=\sigma(\{0,\ldots,t/\beta-1\})$. Then $\mu(G)=t$ and it is clear from the construction that $f^*\chi_{[t,\infty)}$ and $f\chi_{\mathcal{R}\setminus G}$ are equimeasurable.
    \end{itemize}
    In both cases, it follows that
    \begin{equation*}
        \delta_Xf(t) \leq \norm{f\chi_{\mathcal{R}\setminus G}}_X=\norm{f^*\chi_{[t,\infty)}}_{\overline{X}}.
    \end{equation*}
    
    Next, suppose that $\alpha>0$ and that $\norm{f^*\chi_{[t,\infty)}}_{\overline{X}}<\infty$, otherwise the inequality is trivial. Since it holds that $\alpha=\lim_{s\to\infty }f^*(s)\chi_{[t,\infty)}(s)$, all the statements in  \cite[Theorem~4.16]{MusilovaNekvinda25} are true. In particular, there exists a finite constant $C$ such that
    \begin{equation*}
        \norm{\chi_{[0,\infty)}}_{\overline{X}} = \norm{\chi_{[t,\infty)}}_{\overline{X}} = \norm{\chi_{\mathcal{R} \setminus E}}_X = \norm{\chi_\mathcal{R}}_X = C < \infty,
    \end{equation*}
    for every $t>0$ and every $E \subseteq \mathcal{R}$, $\mu(E) < \infty$ (the first three equalities are due to equimeasurability of the functions in question, the finiteness of $C$ comes from the cited result). Let now $f_0:=\max\{f,\alpha\}$. Then $f^*=f_0^*$ and $f\leq f_0$. Using these facts, the properties \ref{Q1} and \ref{P2} of $\norm{\cdot}_X$ and $\norm{\cdot}_{\overline{X}}$, the calculation in the previous case for the function $f_0-\alpha$, and the properties of the non-increasing rearrangement, we obtain
    \begin{equation*}
        \begin{aligned}
            \delta_Xf(t)&\leq\inf_{\substack{E\subseteq \mathcal{R} \\\mu(E) \leq t}} \norm{f_0 \chi_{\mathcal{R}\setminus E}}_X \\
            &\lesssim\inf_{\substack{E\subseteq \mathcal{R} \\\mu(E) \leq t}} \left( \norm{\left(f_0-\alpha\right) \chi_{\mathcal{R}\setminus E}}_X + \norm{\alpha \chi_{\mathcal{R}\setminus E}}_X\right) \\
            &= \delta_X \left(f_0-\alpha\right)(t) + \alpha C \\
            &\leq \norm{\left(f_0-\alpha\right)^*\chi_{[t,\infty)}}_{\overline{X}} + \alpha \norm{\chi_{[t,\infty)}}_{\overline{X}} \\
            &\leq \norm{f^*\chi_{[t,\infty)}}_{\overline{X}} + \norm{f^*\chi_{[t,\infty)}}_{\overline{X}} \\
            &\lesssim \norm{f^*\chi_{[t,\infty)}}_{\overline{X}}.
        \end{aligned}
    \end{equation*}
\end{proof}

A rather immediate consequence is the following representation of our conditions.

\begin{corollary} \label{Cor:RepreDecays}
    Let $(\mathcal{R}, \mu)$ be resonant, let $\norm{\cdot}_X$ be an r.i.~quasi-Banach function norm, and let $X$ be the corresponding r.i.~quasi-Banach function space. Then, for all $f\in \M_0(\mathcal{R},\mu)$, we have
    \begin{align}
        \lim_{n \to \infty} \sigma_X f(n^{-1}) = 0 &\iff \lim_{n \to \infty} \norm{f^*\chi_{[0,n^{-1})}}_{\overline{X}} = 0, \label{Cor:RepreDecays_i}\\
        \lim_{n \to \infty} \delta_X f(n) = 0 &\iff \lim_{n \to \infty} \norm{f^*\chi_{[n,\infty)}}_{\overline{X}} = 0. \label{Cor:RepreDecays_ii}
    \end{align}
    Furthermore, for every $A \subseteq X$ we have
    \begin{align}
        \lim_{n \to \infty} \sup_{f \in A} \delta_X f(n) = 0 &\iff \lim_{n \to \infty} \sup_{f \in A} \norm{f^*\chi_{[n,\infty)}}_{\overline{X}} = 0 \label{Cor:RepreDecays_iii}
    \end{align}
    and when $(\mathcal{R}, \mu)$ is non-atomic then also
    \begin{equation}
        \lim_{n \to \infty} \sigma_X A(n^{-1}) = \lim_{n \to \infty} \sup_{f \in A} \sigma_X f(n^{-1}) = 0 \iff \lim_{n \to \infty} \sup_{f \in A} \norm{f^*\chi_{[0,n^{-1})}}_{\overline{X}} = 0. \label{Cor:RepreDecays_iv}
    \end{equation}
\end{corollary} 

\begin{proof}
    If all the values $n^{-1}$ are in the range of $\mu$, then \eqref{Cor:RepreDecays_i} and \eqref{Cor:RepreDecays_iv} are immediate. If this is not true, then $(\mathcal{R}, \mu)$ is completely atomic and thus all the terms in \eqref{Cor:RepreDecays_i} are zero for every $f \in \mathcal{M}(\mathcal{R}, \mu)$. This is trivial for the left-hand side; for the right-hand side it follows from the construction of the canonical representation space (see \cite[Definition~3.1]{Pesa25}). 
    
    If all the values $n$ are in the range of $\mu$, then \eqref{Cor:RepreDecays_ii} is immediate, while \eqref{Cor:RepreDecays_iii} follows because the constants in Proposition~\ref{global_norm_rearrangement} are independent of $f$. As for the remaining cases:
    \begin{itemize}[leftmargin=5mm]
        \item If $\mu(R) < \infty$ then all the terms in \eqref{Cor:RepreDecays_ii} and \eqref{Cor:RepreDecays_iii} are clearly zero (just take $n \geq \mu(\mathcal{R})$).
        \item If $(\mathcal{R}, \mu)$ is completely atomic with $\mu(\mathcal{R}) = \infty$ then we may find for every $n \in \mathbb{N}$ some $t \geq n$ in the range of $\mu$ and some $m \in \mathbb{N}$, $m \geq t$. This is sufficient thanks to the monotonicity of $s \mapsto \delta_X f(s)$ and $s \mapsto \lVert f^* \chi_{[s, \infty)} \rVert_X$.
    \end{itemize}
\end{proof}

The restriction to non-atomic measure in the last equivalence is necessary because if the measure is completely atomic, then the left-hand side is zero once $n^{-1}$ is smaller than the measure of a single atom regardless of the set $A$, while on the right-hand side we still have some non-trivial condition, specifically
\begin{equation*}
    \lim_{n \to \infty} \sup_{f \in A} \int_0^{n^{-1}} f^* \: d\lambda  = 0.
\end{equation*}
(Recall that we always consider $\lVert \cdot \rVert_{\overline{X}}$ to be defined as in \cite[Definition~3.1]{Pesa25}).

\begin{remark}
    An important part of Corollary~\ref{Cor:RepreDecays} is that there is no representation for
    \begin{equation} \label{Eq:NonRepresented}
        \lim_{n \to \infty} \delta_X A(n) = 0.
    \end{equation}
     That this condition cannot be inferred from the set $\{f^* ; \; f \in A\}$ can be observed quite easily by considering a set of all translations of a single function defined on $(\mathbb{R}^n, \lambda^n)$. This set clearly violates \eqref{Eq:NonRepresented}, yet the non-increasing rearrangements of all the functions are equal to each other, so $\{f^* ; \; f \in A\}$ is a singleton.
\end{remark}

We may now combine Corollary~\ref{Cor:RepreDecays} with the previously presented results to obtain formulations working with the canonical representation space. We note that the above discussed lack or representation for \eqref{Eq:NonRepresented} means that the results working with this condition cannot be formulated only in the terms of the canonical representation space. We shall not list all the results, selecting only the most important ones.

\begin{remark}
    The first result would be a represented version of Theorem~\ref{Thm:ac_char}, i.e.~that $f \in X_a$ if and only if both
    \begin{align*}
        \lim_{n \to \infty} \norm{f^*\chi_{[0,n^{-1})}}_{\overline{X}} = 0, \\
        \lim_{n \to \infty} \norm{f^*\chi_{[n,\infty)}}_{\overline{X}} = 0,
    \end{align*}    
    where $X$ is assumed to be an r.i.~quasi-Banach function space, but this statement has already been proved directly in \cite[Proposition~4.5]{Pesa25}. It is also worth mentioning that in \cite[Corollary~3.7]{KiwerskiKolwicz22} the authors have shown that when $X$ is an r.i.~Banach function space over $((0, \infty), \lambda)$ and $X_a \neq \{ 0 \}$, then one has
    \begin{equation*}
        \dist (f, X_a) = \lim_{n \to \infty} \lVert f^* \chi_{\Omega_n} \rVert_X,
    \end{equation*} 
    where
    \begin{equation*}
        \Omega_n = (0, n^{-1}) \cup (n, \infty).
    \end{equation*}
    This can be seen as a quantitative version of the characterisation, which notably does not follow from the approach presented in this paper. The authors of  \cite{KiwerskiKolwicz22} also cover the cases of a finite interval and of natural numbers with counting measure but their approach is somewhat different to ours;  for simplicity, we present the case that is both the most illustrative and most compatible with our approach.
\end{remark}

The first result we actually present is the characterisation of compactness using $\mu_{\loc}$ (based on Theorem~\ref{Thm:CharComp}).

\begin{theorem} \label{Thm:CharCompRI}
    Let $(\mathcal{R}, \mu)$ be resonant, let $\norm{\cdot}_X$ be an r.i.~quasi-Banach function norm, and let $X$ be the corresponding r.i.~quasi-Banach function space. Let $A \subseteq X_a$. Then the following statements are equivalent:
    \begin{enumerate}
        \item \label{Thm:CharCompRI_i} $A$ is precompact in $(X, \lVert \cdot \rVert_X)$.
        \item \label{Thm:CharCompRI_ii} The following three conditions are all satisfied:
        \begin{enumerate}
            \item \label{Thm:CharCompRI_iia} $A$ is precompact in $(\mathcal{M}_0, \mu_{\loc} )$,
            \item \label{Thm:CharCompRI_iib} $\lim_{n \to \infty} \sup_{f \in A} \lVert f^* \chi_{[0, n^{-1})}  \rVert_{\overline{X}} = 0$,
            \item \label{Thm:CharCompRI_iic} $\lim_{n \to \infty} \delta_X A(n) = 0$.
        \end{enumerate}
    \end{enumerate} 
\end{theorem}

The second result is a characterisation of convergence that combines Theorems~\ref{Thm:CharConv}, \ref{Thm:ConvAlt}, and \ref{Thm:measureConvChar} into a single less complicated statement. This simplification is due to the fact that we assume a~priori that $(\mathcal{R}, \mu)$ is resonant and that rearrangement-invariance implies both admissibility and the finiteness of $\varphi_X^{\max} = \varphi_X$.

\begin{theorem}
    Let $(\mathcal{R}, \mu)$ be resonant, let $\norm{\cdot}_X$ be an r.i.~quasi-Banach function norm, and let $X$ be the corresponding r.i.~quasi-Banach function space. Let $f_k$ be a sequence of functions in $X$, $f \in X$, and denote $A = \{ f_k; \; k \in \mathbb{N} \} \cup \{ f \}$. Then the following statements are equivalent:
    \begin{enumerate}
        \item $f_k \to f$ in $(X, \lVert \cdot \rVert_X)$.
        \item The following three conditions are all satisfied:
        \begin{enumerate}
            \item $f_k \to f$ in $(\mathcal{M}_0, \mu_{\loc} )$,
            \item $\lim_{n \to \infty} \sup_{f \in A} \lVert f^* \chi_{[0, n^{-1})}  \rVert_{\overline{X}} = 0$,
            \item $\lim_{n \to \infty} \delta_X A(n) = 0$.
        \end{enumerate}
        \item The following three conditions are all satisfied:
        \begin{enumerate}
            \item $f_k \to f$ in measure,
            \item $\lim_{n \to \infty} \sup_{f \in A} \lVert f^* \chi_{[0, n^{-1})}  \rVert_{\overline{X}} = 0$,
            \item $\lim_{n \to \infty} \sup_{f \in A} \lVert f^* \chi_{[n, \infty)}  \rVert_{\overline{X}} = 0$.
        \end{enumerate}
    \end{enumerate}
\end{theorem}

\begin{remark}
    It holds that $f_n \to f$ in measure if and only if $(f_n - f)^*(t) \to 0$ for every $t \in (0, \infty)$.
\end{remark}

\bibliographystyle{dabbrv}
\bibliography{bibliography}
\end{document}